\documentclass[a4paper,10pt]{amsart}
\usepackage[english]{babel}
\usepackage{amsmath,tikz-cd}
\usepackage{amssymb}
\usepackage{mathrsfs}
\usepackage{enumerate}
\usepackage{mathtools,yfonts}
\usepackage{tikz,tkz-euclide}
\usepackage{pgfplots}
\usepackage{hyperref}
\usetikzlibrary{arrows}
\usepackage{makecell}

\usepackage{tcolorbox}
\usepackage[]{algorithm2e}

\newtheorem{thm}{Theorem}[section]
\newtheorem{Lemma}[thm]{Lemma}
\newtheorem{Proposition}[thm]{Proposition}
\newtheorem{Corollary}[thm]{Corollary}

\newtheorem*{thm*}{Theorem}

\theoremstyle{definition}

\newtheorem{Definition}[thm]{Definition}
\newtheorem{Remark}[thm]{Remark}

\definecolor{wwwwww}{rgb}{0.4,0.4,0.4}

\DeclareMathOperator{\Cox}{Cox}
\DeclareMathOperator{\Pic}{Pic}

\DeclareMathOperator{\Eff}{Eff}

\DeclareMathOperator{\ch}{char}

\DeclareMathOperator{\Sing}{Sing}

\DeclareMathOperator{\Sym}{Sym}

\DeclareMathOperator{\rank}{rank}

\DeclareMathOperator{\Nef}{Nef}
\DeclareMathOperator{\Mov}{Mov}

\hypersetup{pdfpagemode=UseNone}
\hypersetup{pdfstartview=FitH}
		
\begin{document}

\title{On the unirationality of quadric bundles}

\author[Alex Massarenti]{Alex Massarenti}
\address{\sc Alex Massarenti\\ Dipartimento di Matematica e Informatica, Universit\`a di Ferrara, Via Machiavelli 30, 44121 Ferrara, Italy}
\email{msslxa@unife.it}

\date{\today}

\begin{abstract}
We prove that a general $n$-fold quadric bundle $\mathcal{Q}^{n-1}\rightarrow\mathbb{P}^{1}$, over a number field, with $(-K_{\mathcal{Q}^{n-1}})^n > 0$ and discriminant of odd degree $\delta_{\mathcal{Q}^{n-1}}$ is unirational, and that the same holds for quadric bundles over an arbitrary infinite field provided that $\mathcal{Q}^{n-1}$ has a point, is otherwise general and $n\leq 5$. As a consequence we get the unirationality of a general $n$-fold quadric bundle $\mathcal{Q}^{h}\rightarrow\mathbb{P}^{n-h}$ with discriminant of odd degree $\delta_{\mathcal{Q}^{h}}\leq 3h+4$, and of any smooth $4$-fold quadric bundle $\mathcal{Q}^{2}\rightarrow\mathbb{P}^{2}$, over an algebraically closed field, with $\delta_{\mathcal{Q}^{2}}\leq 12$. 
\end{abstract}

\maketitle
\setcounter{tocdepth}{1}
\tableofcontents

\section{Introduction}
An $n$-dimensional variety $X$ over a field $k$ is rational if it is
birational to $\mathbb{P}^n_{k}$, while $X$ is unirational if there is
a dominant rational map $\mathbb{P}^n_{k}\dasharrow X$. The L\"uroth
problem, asking whether every unirational variety was rational, dates
back to the second half of the nineteenth century \cite{Lu75}. These
two notions turned out to be equivalent for curves and complex surfaces. Only
in the $1970$s examples of unirational but non rational $3$-folds, over an
algebraically closed field of characteristic zero, were given by
M. Artin and D. Mumford \cite{AM72}, V. Iskovskih and I. Manin
\cite{IM71}, and  C. Clemens and P. Griffiths \cite{CG72}. 

The problem of determining whether a variety is rational or unirational is
in general very hard and unirationality is very poorly understood. For instance, the rationality of certain smooth cubic $4$-folds has been proved only recently \cite{BRS19}, \cite{RS19}, and even today not many examples of unirational non
rational varieties have been worked out. Also due to this difficulty,
unirationality has gradually been replaced by the notion of rational connection
\cite{Ca92}, \cite{KMM92}. A variety $X$ is rationally connected if
two general points of $X$ can be joined by a rational curve. We refer to \cite{Ar05} for a comprehensive survey on the subject.

Moreover, the set of rational points of a unirational variety, over an infinite field, is dense. This fact makes  unirationality over a number field an interesting property not only in birational geometry but also in number theory.

When dealing with unirationality questions ad hoc constructions are often needed. In this paper we introduce instead a unifying strategy that can be applied to quadric bundles and more generally to fibrations $\pi:X\rightarrow Y$ over a field $k$ when $Y$ is unirational and the generic fiber of $\pi$ is unirational over $k(Y)$. 

The heart of our approach lies in the Enriques's unirationality criterion which states that $X$ is unirational if and only if $\pi$ has a unirational multisection. In order to produce such multisection we proceed as follows:
\begin{itemize}
\item[-] either we construct unirational subvarieties of $X$ mapping dominantly onto $Y$ as intersections of special divisors on $X$; or
\item[-] we consider birational transformations $X\dasharrow X'$ where $X'\rightarrow Y'$ is a fibration whose general fiber $F'$ is unirational and such that the strict transform of $F'$ in $X$ maps dominantly onto $Y$. 
\end{itemize}
All the arguments presented in the paper rely on this general approach. More specifically, we apply the above strategy to quadric bundles $X$ over $\mathbb{P}^1$ seen as hypersurfaces embedded in splitting projective bundles. In this case, taking advantage of the toric quotient construction of such splitting projective bundle, we can write down the equation of $X$ in Cox coordinates. In such a way we are able to describe explicitly special unirational subvarieties of $X$ and also the birational transformations that $X$ inherits from the ambient projective bundle.  

Once this is done we investigate quadric bundles over higher dimensional projective spaces by studying their restrictions to a general line and applying to these our results for quadric bundles over $\mathbb{P}^1$.

Rationality of conic bundles has been extensively studied and we have a precise conjectural rationality statement, we refer to \cite{Pr18} for a comprehensive survey. 
Moreover, quadric bundles have recently received great attention especially concerning stable rationality \cite{HKT16}, \cite{To16}, \cite{AO18}, \cite{Sc18}, \cite{BG18}, \cite{HPT18}, \cite{HT19}, \cite{Sc19a}, \cite{Sc19b}, \cite{Pa21}, \cite{BKK21}, \cite{NO22}. We recall that a variety $X$ is stably rational if $X\times \mathbb{P}^m_{k}$ is rational for some $m\geq 0$. Hence, a rational variety is stably rational, and a stably rational variety is unirational. As an application of one of our main results together with \cite{HPT18}, \cite{HT19} we will derive examples of unirational but non stably rational quadric bundles.

On the contrary unirationality is still widely open and not much is known. Classically, conjectures on unirationality of conic bundles take into account the degree of the
discriminant and, mimicking what is known abut rationality, conic
bundles with discriminant of large degree are not expected to be unirational \cite[Section 14.2]{Pr18}. 

We will denote an $n$-fold quadric bundle over $\mathbb{P}^{n-h}$ by $\pi:\mathcal{Q}^{h}\rightarrow\mathbb{P}^{n-h}$. Its discriminant $D_{\mathcal{Q}^{h}}\subset\mathbb{P}^{n-h}$ is the divisor parametrizing the singular quadrics in the fibration $\pi:\mathcal{Q}^{h}\rightarrow\mathbb{P}^{n-h}$. We will denote by $\delta_{\mathcal{Q}^{h}}$ the degree of $D_{\mathcal{Q}^{h}}$. 
   
By \cite[Corollary 8]{KM17} smooth conic bundles such that $(-K_{\mathcal{Q}^1})^2 > 0$ are unirational as soon as they have a point. The positivity condition  $(-K_{\mathcal{Q}^1})^2 > 0$ translates into $\delta_{\mathcal{Q}^1}\leq 7$. We will investigate this problem for higher dimensional quadric bundles.

As a consequence of a result due to C. Tsen \cite{Ts33} and S. Lang \cite[Corollary on page 378]{Lan52} all quadric bundles $\mathcal{Q}^{n-1}\rightarrow\mathbb{P}^{1}$ over an algebraically closed field are rational. Furthermore, J. Koll\'{a}r proved that any smooth quadric bundle $\mathcal{Q}^{n-1}\rightarrow\mathbb{P}^{1}$ over a local field is unirational if and only if it has a point \cite[Corollary 1.8]{Kol99}. We will prove results of this type over more general fields. As a sample over number fields we have the following:

\begin{thm}\label{th_A}
Let $\pi:\mathcal{Q}^{n-1}\rightarrow\mathbb{P}^{1}$ be a general quadric bundle over a number field. If $(-K_{\mathcal{Q}^{n-1}})^n > 0$ and $\delta_{\mathcal{Q}^{n-1}}$ is odd then $\mathcal{Q}^{n-1}$ is unirational.
\end{thm}
Furthermore, in Corollary \ref{cr_A} we will extend Theorem \ref{th_A} to an arbitrary infinite field provided that either $\delta_{\mathcal{Q}^{n-1}}\leq 3n+1$ or $n\leq 5$ and $\mathcal{Q}^{n-1}$ has a point.

Theorem \ref{th_A} and Corollary \ref{cr_A}, whose proofs rely on the strategy outlined in the first part of the introduction, are the core of the paper, and starting from them we will derive several unirationality results for quadric bundles over higher dimensional projective spaces. For instance, by Corollary \ref{cor1b2} for surface quadric bundles over the projective plane we have the following: 

\begin{thm}\label{th_B}
Let $\pi:\mathcal{Q}^{2}\rightarrow\mathbb{P}^{2}$ be a smooth complex $4$-fold quadric bundle. If $\delta_{\mathcal{Q}^2}\leq 12$ then $\mathcal{Q}^{2}$ is unirational.    
\end{thm}  

Moreover, in Corollary \ref{corEn} and Remark \ref{algclo} we extend Theorem \ref{th_B} to quadric bundles $\pi:\mathcal{Q}^{h}\rightarrow\mathbb{P}^{n-h}$ over more general fields and with $h,n-h\geq 2$. 

Theorem \ref{th_B}, together with \cite[Corollary 1.3]{BKK21}, yields that a very general quadric bundle $\pi:\mathcal{Q}^{2}\rightarrow\mathbb{P}^{2}$ with discriminant of degree $10 \leq \delta_{\mathcal{Q}^2}\leq 12$ is unirational but not stably rational. 

For a special class of quadric bundles, namely divisors of bidegree $(d,2)$ in $\mathbb{P}^{n-h}\times \mathbb{P}^{h+1}$, in Propositions \ref{333P1}, \ref{gen333}, Corollary \ref{smooth} and Remark \ref{12-22} we give more refined results. In particular, when $k$ is the field of complex numbers as a consequence of our results and \cite{HPT18}, \cite{HT19} we have that a very general divisor of bidegree $(2,2)$ in $\mathbb{P}^2\times\mathbb{P}^{h+1}$ for $h = 1,2$ is unirational but not stably rational.

The situation is very different when $\delta_{\mathcal{Q}^{n-1}}$ is even. For instance, the quadric bundles, over a real closed field, of a fixed multidegree and without points form a semialgebraic set containing an open ball in the Euclidean topology. However, in Corollary \ref{potD} we prove a unirationality result over a quadratic extension of the base field which in particular implies the potential density of the rational points. Note that Theorem \ref{th_A} and Corollary \ref{potD} imply \cite[Conjecture 1.3]{HT00} for quadric bundles over $\mathbb{P}^1$. 

Finally, in Section \ref{sec_ff} we will apply the techniques introduced in the previous sections to give unirationality results for quadric bundles over finite fields. 
 
\subsection*{Conventions on the base field and terminology} 
All along the paper the base field $k$ will be of characteristic different form two. In Sections \ref{QBnm}, \ref{QBpv} $k$ will be an infinite field while in Section \ref{sec_ff} $k$ will be a finite field. 

Let $X$ be a variety over $k$. When we say that $X$ is rational or unirational, without specifying over which field, we will always mean that $X$ is rational or unirational over $k$. Similarly, we will say that $X$ has a point or contains a variety with certain properties meaning that $X$ has a $k$-rational point or contains a variety defined over $k$ with the required properties. 

\section{Quadric bundles and the Enriques's unirationality criterion}\label{QB}
In this section we introduce the notation and a version of the Enriques's unirationality criterion for quadric bundles.

\begin{Definition}\label{def1}
Let $W$ be a smooth $(n-h)$-dimensional variety, $\mathcal{E}$ a rank $h+2$ vector bundle over $W$, $\overline{\pi}:\mathbb{P}(\mathcal{E})\rightarrow W$ the associated projective bundle with tautological bundle $\mathcal{O}_{\mathbb{P}(\mathcal{E})}(1)$, and $\mathcal{L}$ a line bundle over $W$. A quadratic form with values in $\mathcal{L}$ is a global section
$$
\sigma \in H^0(W,\Sym^2\mathcal{E}^{\vee}\otimes \mathcal{L}) \cong H^0(\mathbb{P}(\mathcal{E}),\mathcal{O}_{\mathbb{P}(\mathcal{E})}(2)\otimes \overline{\pi}^{*}\mathcal{L}).
$$
An $n$-dimensional quadric bundle over $W$ is a variety of the form $\mathcal{Q}^{h} := \{\sigma = 0\}\subset\mathbb{P}(\mathcal{E})$ where $\sigma\in H^0(\mathbb{P}(\mathcal{E}),\mathcal{O}_{\mathbb{P}(\mathcal{E})}(2)\otimes \overline{\pi}^{*}\mathcal{L})$ is a generically non degenerate quadratic form, endowed with the projection $\pi:\mathcal{Q}^{h}\rightarrow W$ induced by $\overline{\pi}$. A conic bundle is a quadric bundle with $h = 1$.

The discriminant of $\pi:\mathcal{Q}^{h}\rightarrow W$ is the divisor $D_{\mathcal{Q}^{h}}\subset W$ where $\sigma$ does not have full rank. When $W = \mathbb{P}^{n-h}$ the discriminant $D_{\mathcal{Q}^{h}}\subset \mathbb{P}^{n-h}$ is a hypersurface and we will denote by $\delta_{\mathcal{Q}^{h}} := \deg(D_{\mathcal{Q}^{h}})$ its degree. 
\end{Definition}

\begin{Remark}
Hence, a general fiber of $\pi:\mathcal{Q}^{h}\rightarrow W$ is a quadric hypersurface in $\mathbb{P}^{h+1}$. Note that in Definition \ref{def1} the map $\pi:\mathcal{Q}^{h}\rightarrow W$ is not required to be flat. Often quadric bundles are defined as morphisms $\pi:\mathcal{Q}^{h}\rightarrow W$ whose fibers are isomorphic to quadric hypersurfaces of constant dimension and with $\mathcal{Q}^{h}$ smooth. Then $\pi$ is necessarily flat and there exists a rank $h+2$ vector bundle $\mathcal{E}\rightarrow W$, a line bundle $\mathcal{L}\rightarrow W$, and a section $\sigma\in H^0(\mathbb{P}(\mathcal{E}),\mathcal{O}_{\mathbb{P}(\mathcal{E})}(2)\otimes \overline{\pi}^{*}\mathcal{L})$ as above such that $\mathcal{Q}^{h}$ identifies with the zero locus of $\sigma$ in $\mathbb{P}(\mathcal{E})$ and $\pi = \overline{\pi}_{|\mathcal{Q}^{h}}$ \cite[Theorem 3.2]{Pr18}, \cite[Proposition 1.2]{Bea77}.
\end{Remark}

Let $a_0,\dots,a_{h+1} \in\mathbb{Z}_{\geq 0}$, with $a_0\geq a_1\geq\dots\geq a_{h+1}$, be non negative integers, and consider the simplicial toric variety $\mathcal{T}_{a_0,\dots,a_{h+1}}$ with Cox ring
$$\Cox(\mathcal{T}_{a_0,\dots,a_{h+1}})\cong k[x_0,\dots,x_{n-h},y_0,\dots,y_{h+1}]$$
$\mathbb{Z}^2$-grading given, with respect to a fixed basis $(H_1,H_2)$ of $\Pic(\mathcal{T}_{a_0,\dots,a_{h+1}})$, by the following matrix
$$
\left(\begin{array}{cccccc}
x_0 & \dots & x_{n-h} & y_0 & \dots & y_{h+1}\\
\hline
1 & \dots & 1 & -a_0 & \dots & -a_{h+1}\\ 
0 & \dots & 0 & 1 & \dots & 1
\end{array}\right)
$$
and irrelevant ideal $(x_0,\dots,x_{n-h})\cap (y_0,\dots,y_{h+1})$.
Then
$$\mathcal{T}_{a_0,\dots,a_{h+1}}\cong \mathbb{P}(\mathcal{E}_{a_0,\dots,a_{h+1}})$$
with
$\mathcal{E}_{a_0,\dots,a_{h+1}}\cong\mathcal{O}_{\mathbb{P}^{n-h}}(a_0)\oplus\dots\oplus
\mathcal{O}_{\mathbb{P}^{n-h}}(a_{h+1})$. The secondary fan of $\mathcal{T}_{a_0,\dots,a_{h+1}}$ is as follows
$$
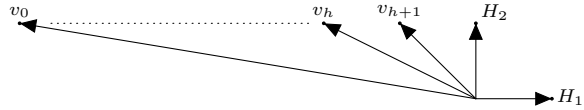
\begin{tikzpicture}[line cap=round,line join=round,>=triangle 45,x=1.0cm,y=1.0cm]
\clip(-7.,-0.2) rectangle (2.,1.3);
\draw [->,line width=0.4pt] (0.,0.) -- (-1.,1.);
\draw [->,line width=0.4pt] (0.,0.) -- (-2.,1.);
\draw [->,line width=0.4pt] (0.,0.) -- (-6.,1.);
\draw [->,line width=0.4pt] (0.,0.) -- (0.,1.);
\draw [->,line width=0.4pt] (0.,0.) -- (1.,0.);
\draw [line width=0.4pt,dotted] (-2.2,1.)-- (-5.6,1.);
\begin{scriptsize}
\draw [fill=black] (-1.,1.) circle (0.5pt);
\draw[color=black] (-1,1.15) node {$v_{h+1}$};
\draw [fill=black] (-2.,1.) circle (0.5pt);
\draw[color=black] (-2,1.15) node {$v_{h}$};
\draw [fill=black] (-6.,1.) circle (0.5pt);
\draw[color=black] (-6,1.15) node {$v_0$};
\draw [fill=black] (0.,1.) circle (0.5pt);
\draw[color=black] (0.25,1.15) node {$H_2$};
\draw [fill=black] (1.,0.) circle (0.5pt);
\draw[color=black] (1.25,0.02) node {$H_1$};
\end{scriptsize}
\end{tikzpicture}
$$
where $H_1 = (1,0)$ corresponds to the sections $x_0,\dots,x_{n-h}$, $H_2 = (0,1)$, and $v_i = (-a_i,1)$ corresponds to the section $y_i$ for $i = 0,\dots,h+1$.

\begin{Definition}\label{splitQB}
A splitting $n$-fold quadric bundle $\pi:\mathcal{Q}^h\rightarrow\mathbb{P}^{n-h}$ is given by an equation of the following form
\stepcounter{thm}
\begin{equation}\label{Cox_g}
\mathcal{Q}^h := \left\lbrace\sum_{0\leq i\leq j\leq h+1}\sigma_{i,j}(x_0,\dots,x_{n-h})y_iy_j = 0\right\rbrace\subset \mathbb{P}(\mathcal{O}_{\mathbb{P}^{n-h}}(a_0)\oplus\dots\oplus
\mathcal{O}_{\mathbb{P}^{n-h}}(a_{h+1}))
\end{equation}
where $\sigma_{i,j}\in k[x_0,\dots,x_{n-h}]_{d_{i,j}}$ is a homogeneous polynomial of degree $d_{i,j}$, and 
\stepcounter{thm}
\begin{equation}\label{compdeg}
d_{0,0}-2a_0 = d_{0,1}-a_0-a_1 =\dots = d_{i,j}-a_i-a_j = \dots = d_{h+1,h+1}-2a_{h+1}.
\end{equation}
The multidegree of a splitting quadric bundle is $(d_{0,0},d_{0,1},\dots, d_{h+1,h+1})\in\mathbb{Z}^{\binom{h+3}{2}}$.
\end{Definition}
Note that the discriminant of a splitting quadric bundle of multidegree $(d_{0,0},d_{0,1},\dots, d_{h+1,h+1})$ has degree $\delta_{\mathcal{Q}^h} = d_{0,0}+\dots + d_{i,i} + \dots + d_{h+1,h+1}$. Without loss of generality we may assume that $a_0\geq a_1\geq\dots\geq a_{h+1}$ so that (\ref{compdeg}) yields $d_{0,0} \geq d_{1,1}\geq \dots\geq d_{h+1,h+1}$. Furthermore, the degrees $d_{i,i}$ of all the $\sigma_{i,i}\neq 0 $ must have the same parity. 
\stepcounter{thm}
\subsection{About the notion of generality}\label{gen_d}
Let $k^{N(n-h,d_{i,j})}$, with $N(n-h,d_{i,j}) = \binom{d_{i,j}+n-h}{n-h}$, be the vector space of degree $d_{i,j}$ homogeneous polynomials in $n-h+1$ variables. Then splitting quadric bundles of multidegree $(d_{0,0},d_{0,1},\dots, d_{h+1,h+1})$ over $\mathbb{P}^{n-h}$ correspond to the elements of 
$$V_{d_{0,0},\dots, d_{h+1,h+1}}^{n-h} = k^{N(n-h,d_{0,0})}\oplus k^{N(n-h,d_{0,1})}\oplus\dots\oplus k^{N(n-h,d_{h+1,h+1})}$$ 
up to multiplication by a non zero scalar. We will say that a splitting quadric bundle $\mathcal{Q}^h$ is general if it corresponds to a general element of $V_{d_{0,0},\dots, d_{h+1,h+1}}^{n-h}$. Furthermore, we will say that a quadric bundle $\pi:\mathcal{Q}^{h}\subset\mathbb{P}(\mathcal{E})\rightarrow\mathbb{P}^{n-h}$ is general if $\mathcal{Q}^{h}_{|L}$ is general as a splitting vector bundle, where $L\subset\mathbb{P}^{n-h}$ is a general line. When referring to a general quadric bundle satisfying certain properties we will mean that the quadric bundle is general among those satisfying the required properties. 

We would like to stress that since all the proofs presented in the paper are constructive it is possible given the equation cutting out $\mathcal{Q}^h\subset\mathbb{P}(\mathcal{E})$ to establish whether or not $\mathcal{Q}^h$ is general in the required sense. 

\begin{Remark}\label{Gro}
By the Birkhoff–Grothendieck splitting theorem \cite[Theorem 4.1]{HM82} all quadric bundles $\mathcal{Q}^{n-1}\rightarrow\mathbb{P}^1$ are splitting.  
\end{Remark}

\subsection{From quadric bundles over $\mathbb{P}^{n-h}$ to quadric bundles over $\mathbb{P}^{1}$}\label{Pn--->P1}
Let $\pi:\mathcal{Q}^{h}\rightarrow\mathbb{P}^{n-h}$ be a quadric bundle $\mathcal{Q}^{h}\subset\mathbb{P}(\mathcal{E})$. Take a point $p\in\mathbb{P}^{n-h}$ and set $Q_p = \pi^{-1}(p)$. Let $\pi_p:\mathbb{P}^{n-h}\dasharrow\mathbb{P}^{n-h-1}$ be the projection from $p$, $W$ the blow-up of $\mathbb{P}^{n-h}$ at $p$, $\widetilde{\mathcal{Q}}^{h}$ the blow-up of $\mathcal{Q}^{h}$ along $Q_p$, and $\widetilde{\pi}_p:W\rightarrow\mathbb{P}^{n-h-1}$ the morphism induced by $\pi_p$. By the universal property of the blow-up \cite[Chapter II, Proposition 7.14]{Ha77} the morphism $\pi:\mathcal{Q}^{h}\rightarrow\mathbb{P}^{n-h}$ induces a morphism $\widetilde{\pi}:\widetilde{\mathcal{Q}}^{h}\rightarrow W$. Note that a general line $L_p$ through $p$ intersects $D_{\mathcal{Q}^{h}}$ in $\delta_{\mathcal{Q}^{h}}$ points counted with multiplicity. Set $\mathcal{Q}^{h}_{L_p} = \pi^{-1}(L_p)$. Then, $\mathcal{Q}^{h}_{L_p}\rightarrow L_p\cong\mathbb{P}^1$ is a quadric bundle $\mathcal{Q}^{h}_{L_p}\subset\mathbb{P}(\mathcal{E}_{|L_p})$, and hence by Remark \ref{Gro} it is splitting. We sum up the situation in the following diagram
$$
\begin{tikzcd}
\widetilde{\mathcal{Q}}^h \arrow[r] \arrow[d, "\widetilde{\pi}"'] & \mathcal{Q}^h \arrow[d, "\pi"]              &                    \\
W \arrow[r] \arrow[rr, "\widetilde{\pi}_p"', bend right=30]          & \mathbb{P}^{n-h} \arrow[r, "\pi_p", dashed] & \mathbb{P}^{n-h-1}.
\end{tikzcd}  
$$
The generic fiber $\widetilde{\mathcal{Q}}^h_{\eta}$ of $\widetilde{\pi}_p\circ\widetilde{\pi}:\widetilde{\mathcal{Q}}^h\rightarrow \mathbb{P}^{n-h-1}$ is a quadric bundle $\widetilde{\mathcal{Q}}^h_{\eta}\rightarrow\mathbb{P}^{1}_{F}$ over $F = k(t_1,\dots,t_{n-h-1})$. Note that $\delta_{\widetilde{\mathcal{Q}}^h_{\eta}} = \delta_{\mathcal{Q}^{h}}$.

\begin{Lemma}\label{gfib}
If $\widetilde{\mathcal{Q}}^h_{\eta}$ is $F$-unirational then $\mathcal{Q}^h$ is $k$-unirational.  
\end{Lemma}
\begin{proof}
If $\widetilde{\mathcal{Q}}^h_{\eta}$ is $F$-unirational then $\widetilde{\mathcal{Q}}^h$ is $k$-unirational, and since $\widetilde{\mathcal{Q}}^h$ is a blow-up of $\mathcal{Q}^h$ we conclude that $\mathcal{Q}^h$ is $k$-unirational as well.  
\end{proof}

The following fundamental results will be our main tools to investigate the birational geometry of quadric bundles over the projective space. 

\begin{Remark}(Lang's theorem)\label{lang}
Fix a real number $r\in\mathbb{R}_{\geq 0}$. A field $k$ is $C_r$ if and only if  every homogeneous polynomial $f\in k[x_0,\dots,n_n]_d$ of degree $d > 0$ in $n+1$ variables with $n+1 > d^r$ has a non trivial zero in $k^{n+1}$.  

If $k$ is a $C_r$ field, $f_1,\dots, f_s\in k[x_0,\dots,n_n]_d$ are homogeneous polynomials of the same degree and $n+1 > sd^r$ then $f_1,\dots,f_s$ have a non trivial common zero in $k^{n+1}$ \cite[Proposition 1.2.6]{Poo17}. Furthermore, if $k$ is $C_r$ then $k(t)$ is $C_{r+1}$ \cite[Theorem 1.2.7]{Poo17}.  

Let $\pi:\mathcal{Q}^{h}\rightarrow \mathbb{P}^{n-h}$ be a quadric bundle over a $C_r$ field $k$ with generic fiber $\mathcal{Q}^{h}_{\eta}$. If $h > 2^{r+n-h}-2$ we have that $\mathcal{Q}^{h}_{\eta}$ has an $F$-point, where $F = k(t_1,\dots,t_{n-h})$. Projecting from such $F$-point we see that $\mathcal{Q}^{h}_{\eta}$ is $F$-rational and hence, arguing as in the proof of Lemma \ref{gfib}, we get that $\mathcal{Q}^{h}$ is rational.   
\end{Remark}

\begin{Remark}(Chevalley–Warning's theorem)\label{C-W}
Let $k$ be a finite field and $f_1,\dots,f_s$ homogeneous polynomials in $n+1$ variables of degree $d_1,\dots,d_s$ with coefficients in $k$. If $n+1 > d_1 + \dots +d_s$ then $f_1,\dots,f_s$ have a non trivial common zero in $k^{n+1}$ \cite{Che35}, \cite{War35}.  
\end{Remark}

Furthermore, we will extensively make use of a unirationality criterion due to F. Enriques \cite[Proposition 10.1.1]{IP99}.  

\begin{Proposition}\label{Enr}
Let $\pi:\mathcal{Q}^h\rightarrow W$ be a quadric bundle over a unirational variety $W$. Then $\mathcal{Q}^h$ is unirational if and only if there exists a unirational subvariety $Z\subset \mathcal{Q}^h$ such that $\pi_{|Z}:Z\rightarrow W$ is dominant. 
\end{Proposition}
\begin{proof}
Assume that $\mathcal{Q}^{h}$ is unirational. Then there exists a dominant rational map $\psi:\mathbb{P}^n\dasharrow\mathcal{Q}^{h}$. If $H\subset\mathbb{P}^n$ is a general $(n-h)$-plane then $Z = \overline{\psi(H)}\subset \mathcal{Q}^{h}$ is unirational and transverse to the fibration $\pi:\mathcal{Q}^h\rightarrow W$ that is $\pi_{|Z}:Z\rightarrow W$ is dominant.      

Now, assume that there exists a unirational subvariety $Z\subset \mathcal{Q}^h$ such that $\pi_{|Z}:Z\rightarrow W$ is dominant. Consider the fiber product 
\[
  \begin{tikzpicture}[xscale=2.6,yscale=-1.3]
    \node (A0_0) at (0, 0) {$\mathcal{Q}^h\times_{W}Z$};
    \node (B) at (1, 0) {$\mathcal{Q}^h$};
    \node (A1_0) at (0, 1) {$Z$};
    \node (A1_1) at (1, 1) {$W$};
    \path (A0_0) edge [->]node [auto] {$\scriptstyle{}$} (B);
    \path (B) edge [->]node [auto] {$\scriptstyle{\pi}$} (A1_1);
    \path (A1_0) edge [->]node [auto] {$\scriptstyle{\pi_{|Z}}$} (A1_1);
    \path (A0_0) edge [->]node [auto] {$\scriptstyle{}$} (A1_0);
  \end{tikzpicture}
\]
and note that $\mathcal{Q}^h\times_{W}Z\rightarrow Z$ is a quadric bundle admitting a rational section $Z\dasharrow \mathcal{Q}^h\times_{W}Z$. Such rational section yields a point of the quadric $\mathcal{Q}:=\mathcal{Q}^h\times_{W}Z$ over the function field $k(Z)$, and by projecting from this point we get that $\mathcal{Q}$ is rational over $k(Z)$, and hence $\mathcal{Q}^h\times_{W}Z$ is birational to $Z\times\mathbb{P}^h$ over $k$. Since $Z$ is unirational then $\mathcal{Q}^h\times_{W}Z$ is unirational, and since $\mathcal{Q}^h$ is dominated by $\mathcal{Q}^h\times_{W}Z$ we conclude that $\mathcal{Q}^h$ is unirational as well.  
\end{proof}

\begin{Remark}\label{not0}
Consider a quadric bundle of the form (\ref{Cox_g}) and assume that one of the $\sigma_{i,i}$, say $\sigma_{0,0}$, is identically zero. As noticed in \cite[Definition 21]{Sc19a} $\{y_1=\dots = y_{h+1} = 0\}$ yields a rational section of $\mathcal{Q}^{h}\rightarrow\mathbb{P}^{n-h}$ and hence $\mathcal{Q}^{h}$ is rational. Furthermore, if $d_{i,i} < 0$ then $\sigma_{i,i} = 0$ and hence we have, from the previous discussion, that $\mathcal{Q}^{h}$ is rational. So all through the paper we will assume that $d_{i,i}\geq 0$ and $\sigma_{i,i}\neq 0$ for all $i = 0,\dots, h+1$. 
\end{Remark}

Next we derive some explicit formulas for the anti-canonical divisor of a quadric bundle. 

\begin{Proposition}\label{ac_emb}
Let $\pi:\mathcal{Q}^h\subset\mathbb{P}(\mathcal{E})\rightarrow\mathbb{P}^{n-h}$ be a quadric bundle with discriminant of degree $\delta_{\mathcal{Q}^h}$. Then 
$$-K_{\mathcal{Q}^h} = \frac{(n-h+1)(h+2)-hc_1(\mathcal{E})-\delta_{\mathcal{Q}^h}}{h+2} \overline{H}_1+h\overline{H}_2$$
where $\overline{H}_1 = \pi^{*}\mathcal{O}_{\mathbb{P}^{n-h}}(1)_{|\mathcal{Q}^h}$ and $\overline{H}_2 = \mathcal{O}_{\mathbb{P}(\mathcal{E})}(1)_{|\mathcal{Q}^h}$.
\end{Proposition}
\begin{proof}
First note that $\mathcal{Q}^h\subset\mathbb{P}(\mathcal{E})$ is a divisor of class 
\stepcounter{thm}
\begin{equation}\label{cls}
\mathcal{Q}^h\sim\frac{\delta_{\mathcal{Q}^h}-2c_1(\mathcal{E})}{h+2}H_1 + 2 H_2
\end{equation}
where $H_1$ is the pull-back to $\mathbb{P}(\mathcal{E})$ of $\mathcal{O}_{\mathbb{P}^{n-h}}(1)$ and $H_2$ is the divisor class corresponding to $\mathcal{O}_{\mathbb{P}(\mathcal{E})}(1)$. To conclude it is enough to use the relative Euler's sequence on $\mathbb{P}(\mathcal{E})$ and the adjunction formula.
\end{proof}

\begin{Proposition}\label{topsi}
Let $\pi:\mathcal{Q}^h\subset\mathbb{P}(\mathcal{E})\rightarrow\mathbb{P}^{n-h}$ be a quadric bundle with discriminant of degree $\delta_{\mathcal{Q}^h}$. Denote by $c_i = c_i(\mathcal{E})$ the Chern classes of $\mathcal{E}$, and define recursively the polynomials 
$$
g_i(c_1,\dots,c_i) = c_1g_{i-1}-c_2g_{i-2}+\dots +(-1)^{i-1}c_ig_0
$$
setting $g_{0} = 1$. Then 
$$
(-K_{\mathcal{Q}^{h}})^n = \sum_{i=0}^{n-h}\binom{n}{n-h-i} \left(\frac{\delta_{\mathcal{Q}^h}-2c_1}{h+2}g_{i-1}+2g_i\right) \left(\frac{(n-h+1)(h+2)-hc_1-\delta_{\mathcal{Q}^h}}{h+2}\right)^{n-h-i}h^{h+i}.
$$ 
\end{Proposition}
\begin{proof}
Note that $H_1^{i} = 0$ for $i > n-h$ and $H_1^{n-h}H_2^{h+1} = 1$. Now, set $
g_i(c_1,\dots, c_i) := H_1^{n-h-i}H_2^{h+i+1}$. Since 
$$
H_1^{n-h-i}H_2^{h+i+1} = c_1H_1^{n-h-i+1}H_2^{h+i}-c_2H_{1}^{n-h-i+2}H_2^{h+i-1}+\dots +(-1)^{i-1}c_iH_1^{n-h}H_2^{h+1}
$$
we get that $g_i = c_1g_{i-1}-c_2g_{i-2}+\dots +(-1)^{i-1}c_ig_0$ and so the $g_i$ can be computed recursively from $g_0 = H_1^{n-h}H_2^{h+1} = 1$. Therefore, plugging-in (\ref{cls}) we have
$$
\overline{H}_1^{n-h-j}\overline{H}_2^{h+j} = H_1^{n-h-j}H_2^{h+j}\left(\frac{\delta_{\mathcal{Q}^h}-2c_1(\mathcal{E})}{h+2}H_1 + 2 H_2\right) = \frac{\delta_{\mathcal{Q}^h}-2c_1}{h+2}g_{j-1} + 2g_j
$$
where we set $g_{-1} = 0$. Finally, Proposition \ref{ac_emb} yields the claim. 
\end{proof}

\begin{Remark}\label{topsiR}
Let $\pi:\mathcal{Q}^{n-1}\rightarrow\mathbb{P}^1$ be a quadric bundle with discriminant of degree $\delta_{\mathcal{Q}^{n-1}}$. Then Proposition \ref{topsi} yields 
$$
(-K_{\mathcal{Q}^{n-1}})^n = \frac{2n(n-1)^{n-1}}{n+1}(2n+2-nc_1+c_1-\delta_{\mathcal{Q}^{n-1}}) + (n-1)^n\frac{\delta_{\mathcal{Q}^{n-1}}+2nc_1}{n+1} = (n-1)^{n-1}(4n-\delta_{\mathcal{Q}^{n-1}}).
$$
For $n = 2$, that is when $\mathcal{Q}^1$ is a conic bundle, we get that $(-K_{\mathcal{Q}^1})^2 = 8-\delta_{\mathcal{Q}^1}$ \cite[Page 5]{Ko17}.
\end{Remark}

\section{Divisors in products of projective spaces}\label{QBnm}
In this section we study the unirationality of divisors of bidegree $(d,2)$ in $\mathbb{P}^1\times\mathbb{P}^n$. This will be crucial in the proof of Theorem \ref{th_A}.

\begin{Proposition}\label{Ott_G}
Let $\pi:\mathcal{Q}^{n-1}\rightarrow \mathbb{P}^{1}$ be a divisor of bidegree $(d,2)$ with $1 < d < n$ defined by an equation of the form
$$
\mathcal{Q}^{n-1} = \left\lbrace \sum_{i=0}^d x_0^{d-i}x_1^i f_i = 0\right\rbrace\subset\mathbb{P}^{1}_{(x_0,x_1)}\times\mathbb{P}^{n}_{(y_0,\dots,y_n)}
$$
with $f_i\in k[y_0,\dots,y_n]_2$. Consider the matrix 
$$
M_{(z_0,\dots,z_{d-1})} = 
\left(
\begin{array}{ccc}
0 & z_0 & f_0\\ 
-z_0 & z_1 & f_1\\ 
\vdots & \vdots & \vdots\\ 
-z_{d-2} & z_{d-1} & f_{d-1}\\ 
-z_{d-1} & 0 & f_d
\end{array} 
\right)
$$
and let $X_{(\overline{z}_0,\dots,\overline{z}_{d-1})} = \{\rank(M_{(\overline{z}_0,\dots,\overline{z}_{d-1})}) < 3\}\subset\mathbb{P}^n$ be the complete intersection of $d-1$ quadrics defined by the $3\times 3$ minors of $M_{(\overline{z}_0,\dots,\overline{z}_{d-1})}$ for a fixed $(\overline{z}_0,\dots,\overline{z}_{d-1})\in k^{d}\setminus\{(0,\dots,0)\}$.

If for some $(\overline{z}_0,\dots,\overline{z}_{d-1})\in k^{d}\setminus\{(0,\dots,0)\}$ the complete intersection $X_{(\overline{z}_0,\dots,\overline{z}_{d-1})}\subset\mathbb{P}^n$ is unirational then $\mathcal{Q}^{n-1}$ is unirational.
\end{Proposition}
\begin{proof}
By \cite[Theorem 1.1 (ii)]{Ott15} if $\mathcal{Q}^{n-1}$ is normal and $\mathbb{Q}$-factorial then it is a Mori dream space and its Mori chamber decomposition is as follows
$$
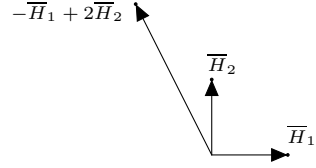
\begin{tikzpicture}[line cap=round,line join=round,>=triangle 45,x=1.0cm,y=1.0cm]
\clip(-2.9,-0.1) rectangle (1.5,2.1);
\draw [->,line width=0.4pt] (0.,0.) -- (1.,0.);
\draw [->,line width=0.4pt] (0.,0.) -- (0.,1.);
\draw [->,line width=0.4pt] (0.,0.) -- (-1.,2.);
\begin{scriptsize}
\draw [fill=black] (1.,0.) circle (0.5pt);
\draw[color=black] (1.19,0.26) node {$\overline{H}_1$};
\draw [fill=black] (0.,1.) circle (0.5pt);
\draw[color=black] (0.14,1.21) node {$\overline{H}_2$};
\draw [fill=black] (-1.,2.) circle (0.5pt);
\draw[color=black] (-1.9,1.9) node {$-\overline{H}_1+2\overline{H}_2$};
\end{scriptsize}
\end{tikzpicture}
$$
where the effective, movable and nef cones are given by 
$$\Eff(\mathcal{Q}^{n-1}) = \Mov(\mathcal{Q}^{n-1}) = \left\langle \overline{H}_1,-\overline{H}_1+2\overline{H}_2\right\rangle, \: \Nef(\mathcal{Q}^{n-1}) = \left\langle \overline{H}_1,\overline{H}_2\right\rangle$$
and the chamber delimited by $\overline{H}_2$ and $-\overline{H}_1+2\overline{H}_2$ corresponds to a small transformation $\mathcal{Q}^{n-1}_{+}$ of $\mathcal{Q}^{n-1}$. The variety $\mathcal{Q}^{n-1}_{+}\subset \mathbb{P}^{d-1}_{(z_0,\dots,z_{d-1})}\times\mathbb{P}^{n}_{(y_0,\dots,y_n)}$ is defined by 
$$
\mathcal{Q}^{n-2}_{+} = \{\rank(M_{(z_0,\dots,z_{d-1})}) < 3\} \subset \mathbb{P}^{d-1}_{(z_0,\dots,z_{d-1})}\times\mathbb{P}^{n}_{(y_0,\dots,y_n)}.
$$
Consider the rational map
$$
\begin{array}{lccc}
\rho: & \mathbb{P}^{1}_{(x_0,x_1)}\times\mathbb{P}^{n}_{(y_0,\dots,y_n)} & \dasharrow & \mathbb{P}^{d-1}_{(z_0,\dots,z_{d-1})}\\
 & ([x_0:x_1],[y_0:\dots:y_n]) & \mapsto & [\rho_0:\dots:\rho_{d-1}]
\end{array}
$$
where $\rho_i([x_0:x_1],[y_0:\dots:y_n]) = x_0^ix_1^{d-i-1}f_0 + x_0^{i+1}x_1^{d-i}f_1 + \dots + x_1^{d-1}f_i$ for $i = 0,\dots,d-1$. The small transformation $\psi:\mathcal{Q}^{n-1}\dasharrow \mathcal{Q}^{n-1}_{+}$ is given by the restriction to $\mathcal{Q}^{n-1}$ of the map
$$
\begin{array}{ccc}
\mathbb{P}^{1}_{(x_0,x_1)}\times\mathbb{P}^{n}_{(y_0,\dots,y_n)} & \dasharrow & \mathbb{P}^{d-1}_{(z_0,\dots,z_{d-1})}\times\mathbb{P}^{n}_{(y_0,\dots,y_n)}\\
([x_0:x_1],[y_0:\dots:y_n]) & \mapsto & (\rho([x_0:x_1],[y_0:\dots:y_n]),[y_0:\dots :y_n]).
\end{array}
$$
We sum-up the situation in the following diagram
$$
\begin{tikzcd}
                           & \mathcal{Q}^{n-1} \arrow[r, "\psi", dashed] \arrow[ld, "\pi"'] & \mathcal{Q}^{n-1}_+ \arrow[rd, "\widetilde{\pi}"] &                                \\
{\mathbb{P}^1_{(x_0,x_1)}} &                                                              &                                                 & {\mathbb{P}^2_{(z_0,\dots,z_{d-1})}}
\end{tikzcd}
$$
where $\widetilde{\pi}:\mathcal{Q}^{n-1}_{+}\rightarrow\mathbb{P}^2_{(z_0,\dots,z_{d-1})}$ is the restriction of the first projection. Now, by hypotheses there exists a fiber $X_{(\overline{z}_0,\dots,\overline{z}_{d-1})} = \widetilde{\pi}^{-1}([\overline{z}_0:\dots :\overline{z}_{d-1}])\subset\mathbb{P}^n$ of $\widetilde{\pi}$ that is unirational. We have $\overline{z}_i\neq 0$ for some $i$, say $\overline{z}_0\neq 0$. So
$$
X_{(\overline{z}_0,\dots,\overline{z}_{d-1})} = \{\overline{z}_0^2f_1 +(\overline{z}_1^2-\overline{z}_0\overline{z}_i)f_0-\overline{z}_0\overline{z}_1f_i = 0; \text{ for } i = 2,\dots,d\}\subset \mathbb{P}^{n}_{(y_0,\dots,y_n)}. 
$$
The strict transform $\widetilde{X}_{(\overline{z}_0,\dots,\overline{z}_{d-1})}\subset\mathcal{Q}^{n+1}$ of $X_{(\overline{z}_0,\dots,\overline{z}_{d-1})}$ via $\psi$ is cut out in $\mathbb{P}^{1}_{(x_0,x_1)}\times\mathbb{P}^{n}_{(y_0,\dots,y_n)}$ by the equation of $\mathcal{Q}^{n+1}$ together with the relations coming from $\rank(\widetilde{M}_{(\overline{z}_0,\dots,\overline{z}_{d-1})}) < 2$ where 
$$
\widetilde{M}_{(\overline{z}_0,\dots,\overline{z}_{d-1})} = 
\left(
\begin{array}{cccc}
x_1^{d-1}f_0 & x_0x_1^{d-2}f_0+x_1^{d-1}f_1 & \dots & x_0^{d-1}f_0+x_0^{d-2}x_1f_1+\dots +x_1^{d-1}f_{d-1}\\ 
\overline{z}_0 & \overline{z}_1 & \dots & \overline{z}_{d-1}
\end{array} 
\right).
$$
Then  
$$
\widetilde{X}_{(\overline{z}_0,\dots,\overline{z}_{d-1})} \cap \{x_1\neq 0\}= 
\left\lbrace
\begin{array}{l}
\overline{z}_{d-1}x_1^{d-1}f_0 - \overline{z}_0(x_0^{d-1}f_0+x_0^{d-2}x_1f_1+\dots + x_1^{d-1}f_{d-1}) = 0;\\ 
\vdots \\ 
\overline{z}_1x_1f_0 - \overline{z}_0(x_0f_0+x_1f_1) = 0;\\
x_0^{d}f_0 +\dots + x_1^{d}f_d = 0.
\end{array} 
\right.
$$
Therefore, $\widetilde{X}_{(\overline{z}_0,\dots,\overline{z}_{d-1})}$ is unirational and $\pi_{|\widetilde{X}_{(\overline{z}_0,\dots,\overline{z}_{d-1})}}:\widetilde{X}_{(\overline{z}_0,\dots,\overline{z}_{d-1})}\rightarrow\mathbb{P}^1_{(x_0,x_1)}$ is dominant. Hence, Proposition \ref{Enr} yields that $\mathcal{Q}^{n-1}$ is unirational as well.
\end{proof}

\begin{Proposition}\label{333P1}
Let $\pi:\mathcal{Q}^{n-1}\rightarrow \mathbb{P}^{1}$ be a divisor of bidegree $(3,2)$ in $\mathbb{P}^{1}\times\mathbb{P}^{n}$ with a point and otherwise general. If $n\geq 4$ then $\mathcal{Q}^{n-1}$ is unirational.
\end{Proposition}
\begin{proof}
We will adopt the notation of Proposition \ref{Ott_G}. Up to a change of coordinates we may assume that $p = ([0:1],[0:\dots:0:1])\in \mathcal{Q}^{n-1}$, so that $f_d$ must be of the form 
$$f_d = A(y_0,\dots,y_{n-1})+y_nL(y_0,\dots,y_{n-1})$$ 
with $A\in k[y_0,\dots,y_{n-1}]_2$ and $L\in k[y_0,\dots,y_{n-1}]_1$. The $f_i$ are otherwise general.   
Note that 
$$
\psi(p) = ([f_0(0,\dots,0,1):f_1(0,\dots,0,1):\dots: f_{d-1}(0,\dots,0,1)],[0:\dots:0:1])
$$ 
and since the $f_i$ are general the point $q = \psi(p)\in \mathcal{Q}^{n-1}_{+}$ is well-defined. Now, set $[\overline{z}_0:\dots:\overline{z}_{d-1}] = \widetilde{\pi}(q) = [f_0(0,\dots,0,1):f_1(0,\dots,0,1):\dots:f_{d-1}(0,\dots,0,1)]$. By Proposition \ref{Ott_G} to conclude it is enough to prove that $X_{(\overline{z}_0,\dots,\overline{z}_{d-1})}\subset\mathbb{P}^n$ is unirational. 

Since the $f_i$ are general the variety $X_{(\overline{z}_0,\dots,\overline{z}_{d-1})}\subset\mathbb{P}^n$ is a smooth complete intersection of two quadrics with a point $q\in X_{(\overline{z}_0,\dots,\overline{z}_{d-1})}$. If $n = 4$ then $X_{(\overline{z}_0,\dots,\overline{z}_{2})}$ is a del Pezzo surface of degree four with a point and hence it is unirational \cite[Summary 9.4.12]{Poo17}. If $\ch(k) = 0$ the unirationality of a smooth complete intersection of two quadrics with a point for $n\geq 4$ follows from \cite[Proposition 2.3]{CSS87}.

If $n> 5$ and $\ch{k} > 0$ we take general hyperplane sections of $X_{(\overline{z}_0,\dots,\overline{z}_{d-1})}$ through $q\in X_{(\overline{z}_0,\dots,\overline{z}_{d-1})}$ until we get a surface $S$ with a point. Since $k$ is infinite \cite[Corollary 3.4.14]{FOV99} yields that $S$ is a smooth del Pezzo surface of degree four and hence it is unirational. The strict transform $\widetilde{S}$ of $S$ via $\psi$ is a unirational surface dominating $\mathbb{P}^1_{(x_0,x_1)}$.
\end{proof}

\begin{Corollary}\label{smooth}
Let $\pi:\mathcal{Q}^{n-1}\rightarrow \mathbb{P}^{1}$ be a smooth divisor of bidegree $(3,2)$ in $\mathbb{P}^{1}\times\mathbb{P}^{n}$ over a $C_r$ field $k$ with $\ch(k) = 0$. If $n > 2^r-1$ and $n \geq 4$ then $\mathcal{Q}^{n-1}$ is unirational.
\end{Corollary}
\begin{proof}
Since $n > 2^r-1$ and $k$ is $C_r$ all fibers of $\pi$ have a point. The general fiber of $\pi$ is a smooth quadric and hence the set of its rational point is dense. So, the set of rational points $\mathcal{Q}^{n-1}(k)$ of $\mathcal{Q}^{n-1}$ is also dense. 

Then $\psi(\mathcal{Q}^{n-1}(k))$ is dense in $\mathcal{Q}^{n-1}_{+}$ and since the smooth locus of $\widetilde{\pi}$ is open there is a point $q\in \psi(\mathcal{Q}^{n-1}(k))$ such that the fiber $X_{(\overline{z}_0,\overline{z}_1,\overline{z}_{2})}$ of $\widetilde{\pi}$ through $q$, where $[\overline{z}_0:\overline{z}_1:\overline{z}_{2}] = \widetilde{\pi}(q)$, is smooth at $q$.   

Hence, \cite[Proposition 2.3]{CSS87} yields that $X_{(\overline{z}_0,\overline{z}_1,\overline{z}_{2})}$ is unirational and to conclude it is enough to apply Proposition \ref{Ott_G}.   
\end{proof}

\begin{Remark}\label{12-22}
A divisor of bidegree $(1,2)$ in $\mathbb{P}^{n-h}\times\mathbb{P}^{h+1}$ is rational. Furthermore, if $\mathcal{Q}^{n-1}\subset\mathbb{P}^1\times\mathbb{P}^n$ is a divisor of bidegree $(2,2)$ the variety $X_{(\overline{z}_0,\dots,\overline{z}_{d-1})}$ in Proposition \ref{Ott_G} is a quadric hypersurface. So arguing as in the proof of Proposition \ref{333P1} we get that $\mathcal{Q}^{n-1}$ is unirational provided it has a point. Now, let $\mathcal{Q}^{h}\subset\mathbb{P}^{n-h}\times\mathbb{P}^{h+1}$ be a divisor of bidegree $(2,2)$ having a point $p\in \mathcal{Q}^{h}$, and $L\subset \mathbb{P}^{h+1}$ a general line through the image of $p$. The preimage $\pi_2^{-1}(L)$ of $L$ via the second projection is a divisor of bidegree $(2,2)$ in $\mathbb{P}^{n-h}\times\mathbb{P}^1$ with a point. Hence, by the first part of the remark $\pi_2^{-1}(L)$ is unirational and since it dominates $\mathbb{P}^{n-h}$ Proposition \ref{Enr} yields that $\mathcal{Q}^{h}$ is unirational.         
\end{Remark}

\section{Quadric bundles with positive volume}\label{QBpv}
Let $\pi:\mathcal{Q}^{n-1}\subset\mathcal{T}_{a_0,\dots,a_{n}}\rightarrow\mathbb{P}^1$ be a quadric bundle of multidegree $(d_{0,0},\dots,d_{n,n})$ cut out by an equation as in (\ref{Cox_g}) and set
$$
\mathcal{Q}_j^{n-j-2} = \mathcal{Q}^{n-1}\cap\{y_0 = \dots = y_{j} = 0\}\subset \mathcal{T}_{a_{j+1},\dots,a_{n}}
$$
for $j = 0,\dots,n-2$. Then $\mathcal{Q}_j^{n-j-2}\rightarrow\mathbb{P}^1$ is a quadric bundle with discriminant of degree $
\delta_{\mathcal{Q}_j^{n-j-2}} = d_{j+1,j+1} + \dots + d_{n,n}$. Note that by Proposition \ref{Enr} if $\mathcal{Q}_j^{n-j-2}$ is unirational for some $j = 0,\dots,n-2$ then $\mathcal{Q}^{n-1}$ is also unirational. We will denote by $\sigma = \sigma(x_0,x_1,y_{n-2},y_{n-1},y_{n})$ the polynomial defining $\mathcal{Q}_{n-3}^{1}$ in $\mathcal{T}_{a_{n-2},a_{n-1},a_{n}}$ and by $\rho = \rho(x_0,x_1)$ the discriminant polynomial of the conic bundle $\pi_{|\mathcal{Q}_{n-3}^{1}}:\mathcal{Q}_{n-3}^{1}\rightarrow\mathbb{P}^1$. 

\begin{Lemma}\label{L_1-7}
Let $\pi:\mathcal{Q}^{n-1}\rightarrow\mathbb{P}^1$ be a quadric bundle. If either $\sigma$ splits as the product of two polynomials both depending on $y_{n-2},y_{n-1},y_{n}$; or $\sigma$ does not have a factor depending just on $x_0,x_1$; when $n=2$ the polynomial $\rho$ is not identically zero, and either
\begin{itemize}
\item[(i)] $\delta_{\mathcal{Q}_{n-3}^{1}} \{0,2,4,6\}$ and $\mathcal{Q}_{n-3}^{1}$ has a smooth point; or
\item[(ii)] $\delta_{\mathcal{Q}_{n-3}^{1}} \in \{1,3,5,7\}$;
\end{itemize}
then $\mathcal{Q}^{n-1}$ is unirational.
\end{Lemma}
\begin{proof}
Assume that $\sigma$ splits as the product of two polynomials both depending on $y_{n-2},y_{n-1},y_{n}$. Then $\mathcal{Q}^1_{n-3}$ splits as a union of surfaces $S_1,S_2$ such that $S_1\cap S_1$ is a rational section of $\pi:\mathcal{Q}^{n-1}\rightarrow\mathbb{P}^1$. Now, assume that $\sigma$ does not split as the product of two polynomials both depending on $y_{n-2},y_{n-1},y_{n}$. 

If $\delta_{\mathcal{Q}_{n-3}^{1}} = 0$ then $\mathcal{Q}_{n-3}^{1} = \mathbb{P}^1 \times C$ where $C$ is smooth conic. Consider the projection $\pi_2:\mathcal{Q}_{n-3}^{1}\rightarrow C$. If $p\in \mathcal{Q}_{n-3}^{1}$ is a point then $\mathbb{P}^1\times \pi_2(p)$ is a section of $\mathcal{Q}_{n-3}^{1}\rightarrow\mathbb{P}^1$ and hence $\mathcal{Q}_{n-3}^{1}$ is rational. If $\delta_{\mathcal{Q}_{n-3}^{1}} = 1$ then one of the $\sigma_{i,i}$ must be zero and by Remark \ref{not0} $\mathcal{Q}^{n-1}$ is rational.  

If $\delta_{\mathcal{Q}_{n-3}^{1}} = 2$ then $d_{n-2,n-2} = 2$, $d_{n-1,n-1} = d_{n,n} = 0$. If the point of $\mathcal{Q}_{n-3}^{1}$ lies on $\{y_{n-2} = 0\}$ then $\mathcal{Q}_{n-3}^{1}\cap\{y_{n-2} = 0\}$ is rational. Assume that the point of $\mathcal{Q}_{n-3}^{1}$ does not lie on $\{y_{n-2} = 0\}$. Note that there is a blow-down morphism $\mathcal{Q}_{n-3}^{1}\rightarrow Q\subset\mathbb{P}^{3}$ onto a quadric surface contracting $\mathcal{Q}_{n-3}^{1}\cap\{y_{n-2} = 0\}$. If $\mathcal{Q}_{n-3}^{1}$ has a smooth point then $Q$ also has a smooth point and hence it is rational. 

If $\delta_{\mathcal{Q}_{n-3}^{1}} = 3$, keeping in mind Remark \ref{not0}, we must have $d_{n-2,n-2} = d_{n-1,n-1} = d_{n,n} = 1$. Then $\mathcal{Q}_{n-3}^{1}$ is a surface of bidegree $(1,2)$ in $\mathbb{P}^1\times\mathbb{P}^2$ and hence it is rational by Remark \ref{12-22}.

If $\delta_{\mathcal{Q}_{n-3}^{1}} = 4$ we have the following two possibilities:
$$
(d_{n-2,n-2},d_{n-2,n-1},d_{n-2,n},d_{n-1,n-1},d_{n-1,n},d_{n,n}) \in\{(4,2,2,0,0,0),(2,2,1,2,1,0)\}.
$$
If the point of $\mathcal{Q}_{n-3}^{1}$ lies on $\{y_{n-2} = 0\}$ then $\mathcal{Q}_{n-3}^{1}\cap\{y_{n-2} = 0\}$ is rational. Otherwise we proceed as follows: consider the case 
$$
(d_{n-2,n-2},d_{n-2,n-1},d_{n-2,n},d_{n-1,n-1},d_{n-1,n},d_{n,n}) = (4,2,2,0,0,0),\: (a_{n-2},a_{n-1},a_n) = (2,0,0).
$$
The divisor $\overline{H}_2$ induces the morphism 
$$
\begin{array}{ccc}
\mathcal{Q}_{n-3}^{1} & \longrightarrow & S\subset\mathbb{P}^4_{(\xi_0,\dots,\xi_4)}\\
([x_0:x_1],[y_0:y_1:y_2]) & \mapsto & [x_0^2y_0:x_0x_1y_0:x_1^2y_0,y_1,y_2]
\end{array}
$$
contracting $\mathcal{Q}_{n-3}^{1}\cap\{y_{n-2} = 0\}$ where 
$$S = \{\xi_1^2-\xi_0\xi_2 = P(\xi_0,\dots,\xi_4) = 0\}$$ 
with $P\in k[\xi_0,\dots,\xi_4]_2$. Set $L_v = \{\xi_0 = \xi_1 = \xi_2 = 0\}$. The smooth point of $\mathcal{Q}_{n-3}^{1}$ yields a smooth point $q\in S$. Up to a change of variables we may assume that $q = [1:0:0:0:0]$. The projection of $S$ from $q$ is a cubic surface $S'\subset\mathbb{P}^3_{(z_0,\dots,z_3)}$ containing a line $L$ given by the projection of $L_v$ and singular at a $0$-dimensional cycle of length two supported on $L$ given by the projection of $L_v\cap\{P=0\}$. 

Assume that $S'$ has a triple point $p\in S'$. Up to a change of variables we may assume that $p = [0:0:0:1]$. Then the equations of $S'$ does not depend on $z_3$ and hence the equation of $\mathcal{Q}_{n-3}^{1}\subset\mathcal{T}_{2,0,0}$ does not depend on $y_2$. In particular, $\rho$ is identically zero. If $n =2$ this contradicts the hypotheses. Therefore $S'$ does not have a triple point and hence \cite[Theorem 1.2]{Kol02} yields that $S'$ is unirational. If $n \geq 3$ the singular locus of $\mathcal{Q}_{n-3}^{1}$ yields a section of $\pi:\mathcal{Q}^{n-1}\rightarrow\mathbb{P}^1$ and hence $\mathcal{Q}^{n-1}$ is rational.  

The remaining cases can be worked out using similar arguments. Finally, to get the unirationality of $\mathcal{Q}^{n-1}$ it is enough to apply Proposition \ref{Enr}.
\end{proof}

\begin{Remark}\label{smoothpt}
The existence of a smooth point in Lemma \ref{L_1-7} is necessary. For instance, consider the conic bundle
$$
S = \{x_0^2y_0^2 + x_0^2y_1^2 + x_1^2y_2^2 = 0\}\subset\mathbb{P}^1_{(x_0,x_1)}\times \mathbb{P}^2_{(y_0,y_1,y_2)}
$$
over $k = \mathbb{Q}$. Note that $S$ is singular at $([0:1],[0:0:1])$ and along $\{x_0 = y_2 = 0\}$, and that $S$ is not unirational. Indeed, if it were the set of rational points of $S$ would be dense in $S$. However, setting $x_0 = y_2 = 1$ we get the conic fibration $\{y_0^2+y_1^2 + x_1^2 = 0\}\subset\mathbb{A}^3_{(x_1,y_0,y_1)}$ and the conic $C_{t} = \{y_0^2+y_1^2 + t^2 = 0\}$ does not have points for all $t\neq 0$. 
\end{Remark}

\begin{Proposition}\label{p1uni}
Let $\pi:\mathcal{Q}^{n-1}\rightarrow\mathbb{P}^1$ be a quadric bundle of multidegree $(d_{0,0},\dots,d_{n,n})$. Assume that $\sigma$ does not have a factor depending just on $x_0,x_1$ and that if $n=2$ then $\rho$ is not identically zero. If
\begin{itemize}
\item[(i)] either $d_{n-2,n-2} + d_{n-1,n-1} + d_{n,n} \in \{1,3,5,7\}$; or 
\item[(ii)] $d_{n-2,n-2} + d_{n-1,n-1} + d_{n,n}\in\{0,2,4,6\}$ and $\mathcal{Q}^1_{n-3}$ has a smooth point;  
\end{itemize}
then $\mathcal{Q}^{n-1}$ is unirational. 
\end{Proposition}
\begin{proof}
Since $\delta_{\mathcal{Q}_{n-3}^1} = d_{n-2,n-2} + d_{n-1,n-1} + d_{n,n}\leq 7$ the claim follows from Lemma \ref{L_1-7}.
\end{proof}

\begin{Lemma}\label{sm_h=2}
Let $\pi:\mathcal{Q}^{2}\rightarrow\mathbb{P}^1$ be a smooth quadric bundle. Then $\sigma$ does not have a factor depending just on $x_0,x_1$.
\end{Lemma}
\begin{proof}
If $\sigma$ has a factor depending just on $x_0,x_1$ we may write
$$
\mathcal{Q}^{2} = \left\lbrace y_0\sum_{i=0}^{3}\sigma_{0,i}y_i +\alpha(x_0,x_1)\sum_{1\leq i \leq j \leq 3}\sigma_{i,j}y_iy_j = 0\right\rbrace\subset\mathcal{T}_{a_0,\dots,a_3}.
$$
Hence, $\{y_0 = \alpha(x_0,x_1) = \sum_{i=1}^{3}\sigma_{0,i}y_i = \sum_{1\leq i \leq j \leq 3}\sigma_{i,j}y_iy_j = 0\} \subseteq \Sing(\mathcal{Q}^{2})$
and so $\mathcal{Q}^{2}$ would be singular. 
\end{proof}

\begin{thm}\label{main1}
Let $\pi:\mathcal{Q}^{n-1}\rightarrow\mathbb{P}^1$ be a quadric bundle. Assume that $(-K_{\mathcal{Q}^{n-1}})^n > 0$ and $\delta_{\mathcal{Q}^{n-1}}$ is odd. If either
\begin{itemize}
\item[(i)] $n\leq 5$, $\mathcal{Q}^{n-1}$ has a point and is otherwise general; or
\item[(ii)] $\delta_{\mathcal{Q}^{n-1}}\leq 3n+1$ and $\mathcal{Q}^{n-1}$ is general;
\end{itemize}
then $\mathcal{Q}^{n-1}$ is unirational. Furthermore, if $n \leq 3$ the above statements hold for any smooth quadric bundle. 
\end{thm}
\begin{proof}
By Proposition \ref{topsi} we have that $(-K_{\mathcal{Q}^{n-1}})^n > 0$ if and only if $\delta_{\mathcal{Q}^{n-1}}\leq 4n-1$. If $n = 2$ then $\delta_{\mathcal{Q}^{n-1}}\leq 7$ and we conclude by Proposition \ref{p1uni}. Assume that $n\geq 3$. 

If $d_{n-2,n-2}+d_{n-1,n-1}+d_{n,n} > 7$ then $d_{n-2,n-2} \geq 4$ unless $(d_{n-2,n-2},d_{n-1,n-1},d_{n,n}) = (3,3,3)$. If $d_{n-2,n-2} = 4$ then $d_{n-1,n-1}\geq 2$ and $d_{n,n}\geq 2$. So $\delta_{\mathcal{Q}^{n-1}}\geq 4(n-1)+2+2 = 4n > 4n-1$, a contradiction. If $d_{n-2,n-2} \geq 5$ then $d_{n-1,n-1}\geq 1$ and $d_{n,n}\geq 1$. So $
\delta_{\mathcal{Q}^{n-1}}\geq 5(n-1)+1+1 = 5n-3 > 4n-1$ for $n > 2$, a contradiction. 

Since $\mathcal{Q}^{n-1}$ is general we may assume that $\mathcal{Q}^1_{n-3}$ satisfies the hypotheses of Proposition \ref{p1uni}. Moreover, since $\delta_{\mathcal{Q}^{n-1}}$ is odd all the $d_{i,i}$ are odd and hence $d_{n-2,n-2}+d_{n-1,n-1}+d_{n,n}$ is also odd. Then $d_{n-2,n-2}+d_{n-1,n-1}+d_{n,n}\in\{1,3,5,7\}$ unless $(d_{n-2,n-2},d_{n-1,n-1},d_{n,n}) = (3,3,3)$ and by Proposition \ref{p1uni} $\mathcal{Q}^{n-1}$ is unirational.  

Now, consider the case $(d_{n-2,n-2},d_{n-1,n-1},d_{n,n})= (3,3,3)$. Since $n\leq 5$ the quadric bundle $\mathcal{Q}^{n-1}$ is a divisor of bidegree $(3,2)$ either in $\mathbb{P}^1\times\mathbb{P}^4$ or in $\mathbb{P}^1\times\mathbb{P}^5$ and since by hypothesis it has a point we conclude by applying Proposition \ref{333P1}.  

If $\delta_{\mathcal{Q}^{n-1}}\leq 3n+1$ we will show that the case $(d_{n-2,n-2},d_{n-1,n-1},d_{n,n})= (3,3,3)$ can be ruled out and so we will not need the existence of a point anymore. 

First, note that $d_{0,0}+ \dots + d_{n,n}\leq 3n+1$ implies $d_{n-2,n-2}+d_{n-1,n-2}+d_{n,n}\leq 7$. For $n = 2$ the claim is trivial. We proceed by induction on $n\geq 2$. If $d_{0,0} = 2$ then $d_{i,i}\leq 2$ for all $i = 0,\dots,n$ and then $d_{n-2,n-2}+d_{n-1,n-2}+d_{n,n}\leq 6$. Similarly $d_{0,0} = 1$ yields $d_{n-2,n-2}+d_{n-1,n-2}+d_{n,n}\leq 3$, and $d_{0,0} = 0$ implies $d_{n-2,n-2}+d_{n-1,n-2}+d_{n,n} = 0$. So we may assume that $d_{0,0}\geq 3$.

Now, $d_{0,0}+ \dots + d_{n,n}\leq 3n+1$ and $d_{0,0}\geq 3$ yield that $d_{1,1}+\dots+d_{n,n}\leq 3n+1-d_{0,0}\leq 3n+1-3 = 3(n-1)+1$ and by induction we get that $d_{n-2,n-2}+d_{n-1,n-2}+d_{n,n}\leq 7$. Since $\delta_{\mathcal{Q}^{n-1}}$ is odd to conclude it is enough to apply Proposition \ref{p1uni}. Finally, if $n \leq 3$ it is enough to apply Lemma \ref{sm_h=2} when $n = 3$ and \cite[Corollary 8]{KM17} for $n = 2$.
\end{proof}

\begin{Corollary}\label{potD}
Let $\pi:\mathcal{Q}^{n-1}\rightarrow\mathbb{P}^1$ be a general quadric bundle over an infinite field $k$. If $(-K_{\mathcal{Q}^{n-1}})^n > 0$ then there exists a quadratic extension $k'$ of $k$ such that $\mathcal{Q}^{n-1}$ is unirational over $k'$. Furthermore, if $n \leq 3$ the above statement holds for any smooth quadric bundle. 
\end{Corollary}
\begin{proof}
In (\ref{Cox_g}) set $x_0 = y_0 = \dots = y_{n-2} = 0$, $x_1 = 1$. Then we get a homogeneous polynomial $f(y_{n-1},y_n)$ of degree two with coefficients in $k$. Set $\overline{f}(y_n) = f(1,y_n)$ and let $k'$ be the splitting field of $\overline{f}$ over $k$. Then $\mathcal{Q}^{n-1}$ has a point over $k'$. Arguing as in the proof of Theorem \ref{main1} we have that $d_{n-2,n-2}+d_{n-1,n-1}+d_{n,n}\leq 7$ with the only exception $(d_{n-2,n-2},d_{n-1,n-1},d_{n,n}) = (3,3,3)$. Hence, to conclude it is enough to apply Propositions \ref{333P1} and \ref{p1uni}.
\end{proof}

\begin{Corollary}\label{corEn}
Let $\pi:\mathcal{Q}^h\rightarrow\mathbb{P}^{n-h}$ be a general quadric bundle with discriminant of odd degree $\delta_{\mathcal{Q}^h}$. If either
\begin{itemize}
\item[(i)] $\delta_{\mathcal{Q}^h}\leq 3h+4$; or
\item[(ii)] $\delta_{\mathcal{Q}^h}\leq 4h+3$, $h\leq 4$, $k$ is $C_r$ and $h+2 > 2^{r+n-h-1}$;
\end{itemize}
then $\mathcal{Q}^h$ is unirational. Furthermore, if $h \leq 2$ the above statements hold for any smooth quadric bundle. 
\end{Corollary}
\begin{proof}
By Lemma \ref{gfib} it is enough to prove that the quadric bundle $\widetilde{\mathcal{Q}}^h_{\eta}\rightarrow\mathbb{P}^1_F$ is unirational over $F = k(t_1,\dots,t_{n-h-1})$. Recall that by the construction in Section \ref{Pn--->P1} we have that $\delta_{\widetilde{\mathcal{Q}}^h_{\eta}} = \delta_{\mathcal{Q}^{h}}$. If $\delta_{\widetilde{\mathcal{Q}}^h_{\eta}} = \delta_{\mathcal{Q}^{h}}\leq 3h+4$ we conclude by Theorem \ref{main1}. 

When $\delta_{\mathcal{Q}^h}\leq 4h+3$ and $h\leq 4$ in order to apply Theorem \ref{main1} we need to produce a point on $\widetilde{\mathcal{Q}}^h_{\eta}$. Fix a general hyperplane $H\subset\mathbb{P}^{n-h}$. Keeping in mind the construction in Section \ref{Pn--->P1} notice that $\widetilde{\mathcal{Q}}^h_{\eta}$ has a point if and only if $\mathcal{Q}^h_{|H}\rightarrow H$ has a rational section. Finally, since $k$ is $C_r$ and $h+2 > 2^{r+n-h-1}$ Remark \ref{lang} yields that such a rational section exists.   
\end{proof}

\begin{Remark}\label{algclo}
By Corollary \ref{corEn} we have that a general quadric bundle $\mathcal{Q}^h\rightarrow\mathbb{P}^{n-h}$ with discriminant of odd degree $\delta_{\mathcal{Q}^h} \leq 4h+3$ is unirational in the following cases:
\begin{itemize}
\item[(i)] $k$ is algebraically closed and $n-h = 2$, $h\in\{1,2\}$, or  $n-h = 3$, $h\in\{3,4\}$;
\item[(ii)] $k$ is $C_1$ and $n-h = 1$, $1\leq h\leq 4$, or $n-h = 2$, $h\in\{3,4\}$;
\item[(iii)] $k$ is $C_2$ and $n-h = 1$, $h\in\{3,4\}$.
\end{itemize}
\end{Remark}

\begin{Corollary}\label{corc1q2}
Assume that $k$ is algebraically closed and let $\mathcal{Q}^1\rightarrow\mathbb{P}^{2}$ be a smooth conic bundle. If the discriminant curve $D_{\mathcal{Q}_1}$ has a point $p\in D_{\mathcal{Q}_1}$ of multiplicity $m_p$ and $\delta_{\mathcal{Q}^1}\leq m_p + 7$. Then $\mathcal{Q}^1$ is unirational.  
\end{Corollary}
\begin{proof}
Consider the conic bundle $\widetilde{\mathcal{Q}}^1_{\eta}$ over $F = k(t)$ in Section \ref{Pn--->P1} constructed by projecting from the point $p\in D_{\mathcal{Q}_1}$ of multiplicity $m_p$. Assume that $m_p\geq 2$. The conic bundle $\widetilde{\mathcal{Q}}^1_{\eta}$ has then a multiple fiber $F_p$ defined over $k$ with two $A_1$ singularities on it also defined over $k$. 

By blowing-up these two singular points and then blowing-down the strict transform of $F_p$ we get a conic bundle $\overline{\mathcal{Q}}^1_{\eta}$ with a new reducible fiber whose components are defined over $k$. So we may blow-down one of these components to get a conic bundle $\hat{\mathcal{Q}}^1_{\eta}$ with seven reducible fibers which is therefore unirational. 

If $m_p = 1$ then $\widetilde{\mathcal{Q}}^1_{\eta}$ is already in the form $\overline{\mathcal{Q}}^1_{\eta}$ and to conclude we may proceed as in the last part of the proof for the case $m_p\geq 2$. Let $L\subset\mathbb{P}^2$ be a general line. By Remark \ref{lang} $\mathcal{Q}^1_{L}\rightarrow L$ has a section and hence $\widetilde{\mathcal{Q}}^1_{\eta}$ has a point. Therefore, by Proposition \ref{p1uni} $\widetilde{\mathcal{Q}}^1_{\eta}$ is unirational and to conclude it is enough to apply Lemma \ref{gfib}.        
\end{proof}

\begin{Corollary}\label{cor1b2}
Assume that $k$ is algebraically closed and let $\mathcal{Q}^2\rightarrow\mathbb{P}^{2}$ be a smooth quadric bundle. If $\delta_{\mathcal{Q}^2}\leq 12$ then $\mathcal{Q}^2$ is unirational.  
\end{Corollary}
\begin{proof}
Fix a point $p\in \mathbb{P}^2$ and consider the quadric bundle $\widetilde{\mathcal{Q}}^2_{\eta}$ in Section \ref{Pn--->P1}. Then $\widetilde{\mathcal{Q}}^2_{\eta}$ is smooth and $\delta_{\widetilde{\mathcal{Q}}^2_{\eta}}\leq 12$. Arguing as in the proof of Theorem \ref{main1} we get that the conic bundle $\widetilde{\mathcal{Q}}^1_{\eta,n-3}\subset \widetilde{\mathcal{Q}}^2_{\eta}$ as defined in the beginning of Section \ref{QBpv} has discriminant of degree $\delta_{\widetilde{\mathcal{Q}}^1_{\eta,n-3}}\leq 8$. Now, $\widetilde{\mathcal{Q}}^1_{\eta,n-3}$ spreads to a conic bundle $\mathcal{Q}^1\rightarrow\mathbb{P}^2$ which is contained in $\mathcal{Q}^2$.

We distinguish two cases: either $D_{\mathcal{Q}^1}$ contains $p$ with multiplicity $m_p$ and $\delta_{\mathcal{Q}^1} = m_p+7$, or $D_{\mathcal{Q}^1}$ does not contain $p$ and $\delta_{\mathcal{Q}^1} = 8$. In the first case arguing as in the proof of Corollary \ref{cor1b2} we get that $\widetilde{\mathcal{Q}}^1_{\eta,n-3}$ is unirational. In the second case we project from a smooth point $q\in D_{\mathcal{Q}^1}$ and proceeding as in the proof of Corollary \ref{cor1b2} we construct a conic bundle $\hat{\mathcal{Q}}^1_{\eta}$ with seven reducible fibers birational to $\widetilde{\mathcal{Q}}^1_{\eta,n-3}$. Finally, by Corollary \ref{corc1q2} we get that $\widetilde{\mathcal{Q}}^1_{\eta,n-3}$ is unirational and Proposition \ref{Enr} yields the unirationality of $\mathcal{Q}^2$. 
\end{proof}

\begin{Proposition}\label{gen333}
Let $\mathcal{Q}^{h}\subset\mathbb{P}^{n-h}\times \mathbb{P}^{h+1}$ be a divisor of bidegree $(3,2)$. Assume that $k$ is $C_r$ with $h+2 > 2^{r+n-h-1}$ and $h\geq 3$. If either $\mathcal{Q}^{h}$ is general or $\mathcal{Q}^{h}$ is smooth and $\ch(k) = 0$ then $\mathcal{Q}^{h}$ is unirational. 
\end{Proposition}
\begin{proof}
The quadric bundle $\widetilde{\mathcal{Q}}^h_{\eta}$ is a divisor of bidegree $(3,2)$ in $\mathbb{P}^1\times\mathbb{P}^{h+1}$ over $F = k(t_1,\dots,t_{n-h-1})$. Arguing as in the proof of Corollary \ref{corEn} we produce a point of $\widetilde{\mathcal{Q}}^h_{\eta}$. So Proposition \ref{333P1} and Corollary \ref{smooth} imply that $\widetilde{\mathcal{Q}}^h_{\eta}$ is unirational and Lemma \ref{gfib} yields the unirationality of $\mathcal{Q}^{h}$.   
\end{proof}

Finally, we consider quadric bundles $\pi:\mathcal{Q}^{n-1}\rightarrow\mathbb{P}^1$ over a number field. 

\begin{Lemma}\label{CTB}
Let $\pi:\mathcal{Q}^{h}\rightarrow\mathbb{P}^{n-h}$ be a smooth quadric bundle over a number field $k$ with discriminant of odd degree. If $h\geq 3$ then $\mathcal{Q}^{h}$ has a point. 
\end{Lemma}
\begin{proof}
For a place $v$ of $k$ we will denote by $k_v$ the completion of $k$ at $v$. Consider a general line $L\subset \mathbb{P}^{n-h}$ and the quadric bundle $\pi_{|L}:\mathcal{Q}^{h}_{|L}\rightarrow L$. Then $\delta_{\mathcal{Q}^{h}_{|L}}$ is odd and $\mathcal{Q}^{h}_{|L}$ has a point over the reals and hence $\mathcal{Q}^{h}_{|L}(k_v)$ in not empty for all real places of $k$. Therefore, also $\mathcal{Q}^{h}(k_v)$ in not empty for all real places of $k$. 

Now, a general fiber $Q_q^h = \pi^{-1}(q)\subset\mathbb{P}^{h+1}$ is a quadric hypersurface of dimension $h\geq 3$ over $k$. By \cite[Chapter IV, Theorem 6]{Se77} $Q_q^h$ has a point over all the $p$-adic places of $k$. We conclude that $\mathcal{Q}^{h}$ has a point over $k_v$ for all places $v$, both real and $p$-adic, of $k$. Finally, \cite[Proposition 3.9]{CSS87} yields that $\mathcal{Q}^{h}$ has a point over the base field $k$.
\end{proof}

\begin{thm}\label{main_nf}
Let $\pi:\mathcal{Q}^{n-1}\rightarrow\mathbb{P}^1$ be a general quadric bundle over a number field. If $(-K_{\mathcal{Q}^{n-1}})^n > 0$ and $\delta_{\mathcal{Q}^{n-1}}$ is odd then $\mathcal{Q}^{n-1}$ is unirational. Furthermore, if $n \leq 3$ the above statement holds for any smooth quadric bundle. 
\end{thm}
\begin{proof}
Note that the only case in the proof of Theorem \ref{main1} for which the existence of a point is needed is $(d_{n-2,n-2},d_{n-1,n-1},d_{n,n})= (3,3,3)$. In this case we have $d_{n-3,d-3} = d_{n-4,n-4} = 3$ otherwise $\delta_{\mathcal{Q}^{n-1}} > 4n-1$. So $\mathcal{Q}^{3}_{n-4}$ is a divisor of bidegree $(3,2)$ in $\mathbb{P}^1\times\mathbb{P}^4$.

By Lemma \ref{CTB} a general divisor of bidegree $(3,2)$ in $\mathbb{P}^1\times\mathbb{P}^4$ has point, and hence the claim follows by Propositions \ref{Enr}, \ref{333P1}. 
\end{proof}

We sum up the main results of this section as follows:

\begin{Corollary}\label{cr_A}
Let $\pi:\mathcal{Q}^{n-1}\rightarrow\mathbb{P}^{1}$ be a quadric bundle over an infinite field $k$. Assume that $(-K_{\mathcal{Q}^{n-1}})^n > 0$ and $\delta_{\mathcal{Q}^{n-1}}$ is odd. If either 
\begin{itemize}
\item[(i)] $n\leq 5$, $\mathcal{Q}^{n-1}$ has a point and is otherwise general; or
\item[(ii)] $\mathcal{Q}^{n-1}$ is general and $\delta_{\mathcal{Q}^{n-1}}\leq 3n+1$; or
\item[(iii)] $\mathcal{Q}^{n-1}$ is general and $k$ is a number field;
\end{itemize}
then $\mathcal{Q}^{n-1}$ is unirational.
\end{Corollary}
\begin{proof}
It is enough to apply Theorems \ref{main1} and \ref{main_nf}.
\end{proof} 

\begin{Corollary}\label{height}
Let $\pi:\mathcal{Q}^{n-1}\rightarrow\mathbb{P}^1$ be a general quadric bundle over a number field $k$. If $(-K_{\mathcal{Q}^{n-1}})^n > 0$ and $\delta_{\mathcal{Q}^{n-1}}$ is odd then there exists an $\epsilon > 0$ such that for any open subset $U\subset \mathcal{Q}^{n-1}$
$$
N(U,B) = \sharp\{p\in U(k) \: | \: \text{ht}(p)\leq B\} \geq c_{\mathcal{Q}^{n-1}}B^{\epsilon}
$$
for $B\mapsto \infty$, where $c_{\mathcal{Q}^{n-1}}$ depends on $\mathcal{Q}^{n-1}$, and $\text{ht}$ is the multiplicative height \cite[Definition 1.5.4]{BG06}. 
\end{Corollary}
\begin{proof}
By Theorem \ref{main_nf} there is a dominant rational map $\theta:\mathbb{P}^n\dasharrow \mathcal{Q}^{n-1}\subset\mathbb{P}^N$ given by polynomials of a certain degree, say $d$. Then points of $\mathbb{P}^n$ of height at most $B^{\frac{1}{d}}$ are mapped to points of $X$ of height at most $B$. 

Let $V\subset \mathcal{Q}^{n-1}$ be the subset over which $\theta$ is finite. Then the number of points of height at most $B$ of $U$ goes at least as the number of points of height at most $B^{\frac{1}{d}}$ of $\theta^{-1}(U\cap V)$ which in turn goes at least as the number of points of height at most $B^{\frac{1}{d}}$ of $\mathbb{P}^n$ minus the number of points of height at most $B^{\frac{1}{d}}$ of a closed subset $Z\subset \mathbb{P}^n$. 

Now, to conclude it is enough to observe that $N(\mathbb{P}^n,B^{\frac{1}{d}})$ goes as $c(n)B^{\frac{n+1}{d}}$ \cite[Theorem 2.1]{Pe02} while $N(Z_i,B^{\frac{1}{d}})$ goes as $c'(\dim(Z_i),a_i,n)B^{\frac{1}{d}(\dim(Z_i)+\frac{1}{a_i}+\epsilon)}$ where $Z_i$ is an irreducible component of $Z$ of degree $a_i$ and $0 <\epsilon\ll 1$ \cite[Theorem B]{Pi95}.
\end{proof}

Recently, the distribution of rational points on divisors in products of projective spaces has been much investigated, see for instance \cite{BH19}, \cite{BH20}.

\section{Quadric bundles over finite fields}\label{sec_ff}

In this section $k = \mathbb{F}_q$ will be a finite field with $q$ elements. Let $X\subset\mathbb{P}^n$ be a smooth complete intersection of $c$ quadric hypersurfaces and $F_{c-1}(X)\subset\mathbb{G}(c-1,n)$ the variety parametrizing $(c-1)$-dimensional linear subspaces contained in $X$, where $\mathbb{G}(c-1,n)$ is the Grassmannian of $(c-1)$-dimensional linear subspaces of $\mathbb{P}^n$ in its Pl\"ucker embedding.

\begin{Remark}\label{rsub}
Thanks to the study of the geometry and of the canonical class of $F_{c-1}(X)$ in \cite{DM98}, and arguing as in the proof of \cite[Theorem 2.1]{HT21} we have that if $n \geq c(c+1)$ then $F_{c-1}(X)$ has a point and $X$ is rational.
\end{Remark}

\begin{Lemma}\label{uni2q}
Let $X\subset\mathbb{P}^n$ be a complete intersection of two quadric hypersurfaces. Assume that $n\geq 4$. If either $X$ is singular but not a cone or $X$ is smooth and contains a line then $X$ is rational. Furthermore, if $X$ is smooth then $X$ is unirational.  
\end{Lemma}
\begin{proof}
Write $X = Q_1\cap Q_2$ with $Q_i\subset\mathbb{P}^n$ quadric hypersurface. Since $n\geq 4$ Remark \ref{C-W} yields that $X$ has a point $p\in X$. First assume $X$ to be smooth and that there is a line $L$ through $p$ intersecting $X$ in at least three points counted with multiplicity. Then $L$ is contained in both $Q_1$ and $Q_2$. So $L\subset X$ and Remark \ref{rsub} yields that $X$ is rational. If all the lines through $p$ intersect $X$ in at most two points counted with multiplicity then the image of the birational projection of $X$ from $p$ is a cubic hypersurface $Y\subset\mathbb{P}^{n-1}$ with no triple points. Since $n\geq 4$ Remark \ref{C-W} implies that $Y$ has a point and hence by \cite[Theorem 1]{Kol02} $Y$ is unirational. 

If $X$ has a double point $p\in X$ the projection of $X$ from $p$ is a quadric $Y$ which again by Remark \ref{C-W} has a smooth point, and if $X$ has a triple point $p\in X$ the projection of $X$ from $p$ is a hyperplane $Y$. In both cases $Y$ is rational.  
\end{proof}

\begin{Proposition}\label{thm1ff}
Let $\mathcal{Q}^{n-1}\subset\mathbb{P}^1\times \mathbb{P}^n$ be a divisor of bidegree $(d,2)$. If $d = 1$ then $\mathcal{Q}^{n-1}$ is rational. Furthermore, in the same notation of Proposition \ref{Ott_G}, if either $d = 2$ or
\begin{itemize}
\item[(i)] $d = 3$ and for some $[\overline{z}_0:\overline{z}_1:\overline{z}_2]\in\mathbb{P}^2$ the complete intersection $X_{(\overline{z}_0,\overline{z}_1,\overline{z}_2)}$ is not a cone; or
\item[(ii)] $d \geq 4$, $n\geq d(d-1)$ and for some $[\overline{z}_0:\dots:\overline{z}_{d-1}]\in\mathbb{P}^{d-1}$ the complete intersection $X_{(\overline{z}_0,\dots,\overline{z}_{d-1})}$ is smooth;
\end{itemize}
then $\mathcal{Q}^{n-1}$ is unirational. 
\end{Proposition}
\begin{proof}
The case $d\in\{1,2\}$ follows from Remark \ref{12-22}. Assume that $d = 3$. Then $X_{(\overline{z}_0,\overline{z}_1,\overline{z}_2)}$ is a complete intersection of two quadrics and the claim follows from Proposition\ref{Ott_G} and Lemma \ref{uni2q}. If $d \geq 4$, $n\geq d(d-1)$ then $X_{(\overline{z}_0,\dots,\overline{z}_{d-1})}$ is a complete intersection of $d-1$ quadrics and the claim follows from Remark \ref{rsub}.
\end{proof}

\begin{Proposition}\label{mainff}
Let $\mathcal{Q}^{n-1}\rightarrow\mathbb{P}^1$ be a quadric bundle with $\delta_{\mathcal{Q}^{n-1}}\leq 4n-1$. Assume that $\rho$ does not vanish at all points of $\mathbb{P}^1$, $\sigma$ does not have a factor depending just on $x_0,x_1$ and $\mathcal{Q}^1_{n-3}$ has a smooth point; and if $d_{n-4} = d_{n-3} = d_{n-2} = d_{n-1} = d_n = 3$ assume in addition, in the notation of Proposition \ref{Ott_G}, that there exists $[\overline{z}_0:\overline{z}_1:\overline{z}_2]\in\mathbb{P}^2$ such that $X_{(\overline{z}_0,\overline{z}_1,\overline{z}_2)}$ is irreducible and not a cone. Then $\mathcal{Q}^{n-1}$ is unirational.   
\end{Proposition}
\begin{proof}
For the first part it is enough to argue as in the proof of Theorem \ref{main1}. Assume that $d_{n-2} = d_{n-1} = d_n = 3$. Then $n\geq 3$ and $d_{n-4}=d_{n-3} = d_{n-2} = d_{n-1} = d_n = 3$. So $\mathcal{Q}^3_{n-5}\subset\mathbb{P}^1\times\mathbb{P}^4$ is a divisor of bidegree $(3,2)$ and to conclude it is enough to apply Proposition \ref{thm1ff}. 
\end{proof}

\begin{Remark}\label{ptsff}
Over a finite field $k = \mathbb{F}_q$ the conic bundle $S$ in Remark \ref{smoothpt} is unirational. Indeed, the fiber $\{y_0^2+y_1^2+y_2^2 = 0\}$ over $[x_0,x_1] = [1:1]$ is smooth and hence has $q+1$ points which by the description of $\Sing(S)$ in Remark \ref{smoothpt} are smooth points of $S$.
\end{Remark}

\bibliographystyle{amsalpha}
\bibliography{Biblio}

\providecommand{\bysame}{\leavevmode\hbox to3em{\hrulefill}\thinspace}
\providecommand{\MR}{\relax\ifhmode\unskip\space\fi MR }
\providecommand{\MRhref}[2]{%
  \href{http://www.ams.org/mathscinet-getitem?mr=#1}{#2}
}
\providecommand{\href}[2]{#2}
\begin{thebibliography}{CTSSD87}

\bibitem[AM72]{AM72}
M.~Artin and D.~Mumford, \emph{Some elementary examples of unirational
  varieties which are not rational}, Proc. London Math. Soc. (3) \textbf{25}
  (1972), 75--95. \MR{321934}

\bibitem[AO18]{AO18}
H.~Ahmadinezhad and T.~Okada, \emph{Stable rationality of higher dimensional
  conic bundles}, \'{E}pijournal G\'{e}om. Alg\'{e}brique \textbf{2} (2018),
  Art. 5, 13. \MR{3816900}

\bibitem[Ara05]{Ar05}
C.~Araujo, \emph{Rationally connected varieties}, Snowbird lectures in
  algebraic geometry, Contemp. Math., vol. 388, Amer. Math. Soc., Providence,
  RI, 2005, pp.~1--16. \MR{2182887}

\bibitem[Bea77]{Bea77}
A.~Beauville, \emph{Vari\'{e}t\'{e}s de {P}rym et jacobiennes
  interm\'{e}diaires}, Ann. Sci. \'{E}cole Norm. Sup. (4) \textbf{10} (1977),
  no.~3, 309--391. \MR{472843}

\bibitem[BG06]{BG06}
E.~Bombieri and W.~Gubler, \emph{Heights in {D}iophantine geometry}, New
  Mathematical Monographs, vol.~4, Cambridge University Press, Cambridge, 2006.
  \MR{2216774}

\bibitem[BH19]{BH19}
T.~D. Browning and L.~Q. Hu, \emph{Counting rational points on biquadratic
  hypersurfaces}, Adv. Math. \textbf{349} (2019), 920--940. \MR{3944106}

\bibitem[BHB20]{BH20}
T.~D. Browning and D.~R. Heath-Brown, \emph{Density of rational points on a
  quadric bundle in {$\Bbb{P}^3\times \Bbb{P}^3$}}, Duke Math. J. \textbf{169}
  (2020), no.~16, 3099--3165. \MR{4167086}

\bibitem[BKK22]{BKK21}
G.~Bini, G.~Kapustka, and M.~Kapustka, \emph{Symmetric locally free resolutions
  and rationality problems}, Communications in Contemporary Mathematics (2022).

\bibitem[BRS19]{BRS19}
M.~Bolognesi, F.~Russo, and G.~Staglian\`o, \emph{Some loci of rational cubic
  fourfolds}, Math. Ann. \textbf{373} (2019), no.~1-2, 165--190. \MR{3968870}

\bibitem[BvB18]{BG18}
C.~B\"{o}hning and H.~C.~Graf von Bothmer, \emph{On stable rationality of some
  conic bundles and moduli spaces of {P}rym curves}, Comment. Math. Helv.
  \textbf{93} (2018), no.~1, 133--155. \MR{3777127}

\bibitem[Cam92]{Ca92}
F.~Campana, \emph{Connexit\'{e} rationnelle des vari\'{e}t\'{e}s de {F}ano},
  Ann. Sci. \'{E}cole Norm. Sup. (4) \textbf{25} (1992), no.~5, 539--545.
  \MR{1191735}

\bibitem[CG72]{CG72}
C.~H. Clemens and P.~A. Griffiths, \emph{The intermediate {J}acobian of the
  cubic threefold}, Ann. of Math. (2) \textbf{95} (1972), 281--356. \MR{302652}

\bibitem[Che35]{Che35}
C.~Chevalley, \emph{D\'{e}monstration d'une hypoth\`ese de {M}. {A}rtin}, Abh.
  Math. Sem. Univ. Hamburg \textbf{11} (1935), no.~1, 73--75. \MR{3069644}

\bibitem[CTSSD87]{CSS87}
J.~L. Colliot-Th\'{e}l\`ene, J.~J. Sansuc, and P.~Swinnerton-Dyer,
  \emph{Intersections of two quadrics and {C}h\^{a}telet surfaces. {I}}, J.
  Reine Angew. Math. \textbf{373} (1987), 37--107. \MR{870307}

\bibitem[DM98]{DM98}
O.~Debarre and L.~Manivel, \emph{Sur la vari\'{e}t\'{e} des espaces
  lin\'{e}aires contenus dans une intersection compl\`ete}, Math. Ann.
  \textbf{312} (1998), no.~3, 549--574. \MR{1654757}

\bibitem[FOV99]{FOV99}
H.~Flenner, L.~O'Carroll, and W.~Vogel, \emph{Joins and intersections},
  Springer Monographs in Mathematics, Springer-Verlag, Berlin, 1999.
  \MR{1724388}

\bibitem[Har77]{Ha77}
R.~Hartshorne, \emph{Algebraic geometry}, Springer-Verlag, New York-Heidelberg,
  1977, Graduate Texts in Mathematics, No. 52. \MR{0463157}

\bibitem[HKT16]{HKT16}
B.~Hassett, A.~Kresch, and Y.~Tschinkel, \emph{Stable rationality and conic
  bundles}, Math. Ann. \textbf{365} (2016), no.~3-4, 1201--1217. \MR{3521088}

\bibitem[HM82]{HM82}
M.~Hazewinkel and C.~F. Martin, \emph{A short elementary proof of
  {G}rothendieck's theorem on algebraic vectorbundles over the projective
  line}, J. Pure Appl. Algebra \textbf{25} (1982), no.~2, 207--211. \MR{662762}

\bibitem[HPT18]{HPT18}
B.~Hassett, A.~Pirutka, and Y.~Tschinkel, \emph{Stable rationality of quadric
  surface bundles over surfaces}, Acta Math. \textbf{220} (2018), no.~2,
  341--365. \MR{3849287}

\bibitem[HT00]{HT00}
J.~Harris and Y.~Tschinkel, \emph{Rational points on quartics}, Duke Math. J.
  \textbf{104} (2000), no.~3, 477--500. \MR{1781480}

\bibitem[HT19]{HT19}
B.~Hassett and Y.~Tschinkel, \emph{On stable rationality of {F}ano threefolds
  and del {P}ezzo fibrations}, J. Reine Angew. Math. \textbf{751} (2019),
  275--287. \MR{3956696}

\bibitem[HT21]{HT21}
\bysame, \emph{Varieties of planes on intersections of three quadrics}, Eur. J.
  Math. \textbf{7} (2021), no.~2, 613--632. \MR{4256965}

\bibitem[IM71]{IM71}
V.~A. Iskovskih and J.~I. Manin, \emph{Three-dimensional quartics and
  counterexamples to the {L}\"{u}roth problem}, Mat. Sb. (N.S.)
  \textbf{86(128)} (1971), 140--166. \MR{0291172}

\bibitem[IP99]{IP99}
V.~A. Iskovskikh and Y.~G. Prokhorov, \emph{Fano varieties}, Algebraic
  geometry, {V}, Encyclopaedia Math. Sci., vol.~47, Springer, Berlin, 1999,
  pp.~1--247. \MR{1668579}

\bibitem[KM17]{KM17}
J.~Koll\'{a}r and M.~Mella, \emph{Quadratic families of elliptic curves and
  unirationality of degree 1 conic bundles}, Amer. J. Math. \textbf{139}
  (2017), no.~4, 915--936. \MR{3689320}

\bibitem[KMM92]{KMM92}
J.~Koll\'{a}r, Y.~Miyaoka, and S.~Mori, \emph{Rational connectedness and
  boundedness of {F}ano manifolds}, J. Differential Geom. \textbf{36} (1992),
  no.~3, 765--779. \MR{1189503}

\bibitem[Kol99]{Kol99}
J.~Koll\'{a}r, \emph{Rationally connected varieties over local fields}, Ann. of
  Math. (2) \textbf{150} (1999), no.~1, 357--367. \MR{1715330}

\bibitem[Kol02]{Kol02}
\bysame, \emph{Unirationality of cubic hypersurfaces}, J. Inst. Math. Jussieu
  \textbf{1} (2002), no.~3, 467--476. \MR{1956057}

\bibitem[Kol17]{Ko17}
\bysame, \emph{Conic bundles that are not birational to numerical
  {C}alabi-{Y}au pairs}, \'{E}pijournal G\'{e}om. Alg\'{e}brique \textbf{1}
  (2017), Art. 1, 14. \MR{3743104}

\bibitem[Lan52]{Lan52}
S.~Lang, \emph{On quasi algebraic closure}, Ann. of Math. (2) \textbf{55}
  (1952), 373--390. \MR{46388}

\bibitem[Lur75]{Lu75}
J.~Luroth, \emph{Beweis eines {S}atzes \"{u}ber rationale {C}urven}, Math. Ann.
  \textbf{9} (1875), no.~2, 163--165. \MR{1509855}

\bibitem[NO22]{NO22}
J.~Nicaise and J.~C. Ottem, \emph{Tropical degenerations and stable
  rationality}, https://arxiv.org/abs/1911.06138v4, to appear in Duke
  Mathematical Journal, 2022.

\bibitem[Ott15]{Ott15}
J.~C. Ottem, \emph{Birational geometry of hypersurfaces in products of
  projective spaces}, Math. Z. \textbf{280} (2015), no.~1-2, 135--148.
  \MR{3343900}

\bibitem[Pau21]{Pa21}
M.~Paulsen, \emph{On the rationality of quadric surface bundles}, Ann. Inst.
  Fourier (Grenoble) \textbf{71} (2021), no.~1, 97--121. \MR{4275865}

\bibitem[Pey02]{Pe02}
E.~Peyre, \emph{Points de hauteur born\'{e}e et g\'{e}om\'{e}trie des
  vari\'{e}t\'{e}s (d'apr\`es {Y}. {M}anin et al.)}, no. 282, 2002,
  S\'{e}minaire Bourbaki, Vol. 2000/2001, pp.~Exp. No. 891, ix, 323--344.
  \MR{1975184}

\bibitem[Pil95]{Pi95}
J.~Pila, \emph{Density of integral and rational points on varieties}, no. 228,
  1995, Columbia University Number Theory Seminar (New York, 1992), pp.~4,
  183--187. \MR{1330933}

\bibitem[Poo17]{Poo17}
B.~Poonen, \emph{Rational points on varieties}, Graduate Studies in
  Mathematics, vol. 186, American Mathematical Society, Providence, RI, 2017.
  \MR{3729254}

\bibitem[Pro18]{Pr18}
Y.~G. Prokhorov, \emph{The rationality problem for conic bundles}, Uspekhi Mat.
  Nauk \textbf{73} (2018), no.~3(441), 3--88. \MR{3807895}

\bibitem[RS19]{RS19}
F.~Russo and G.~Staglian\`o, \emph{Congruences of 5-secant conics and the
  rationality of some admissible cubic fourfolds}, Duke Math. J. \textbf{168}
  (2019), no.~5, 849--865. \MR{3934590}

\bibitem[Sch18]{Sc18}
S.~Schreieder, \emph{Quadric surface bundles over surfaces and stable
  rationality}, Algebra Number Theory \textbf{12} (2018), no.~2, 479--490.
  \MR{3803711}

\bibitem[Sch19a]{Sc19a}
\bysame, \emph{On the rationality problem for quadric bundles}, Duke Math. J.
  \textbf{168} (2019), no.~2, 187--223. \MR{3909896}

\bibitem[Sch19b]{Sc19b}
\bysame, \emph{Stably irrational hypersurfaces of small slopes}, J. Amer. Math.
  Soc. \textbf{32} (2019), no.~4, 1171--1199. \MR{4013741}

\bibitem[Ser77]{Se77}
J.~P. Serre, \emph{Cours d'arithm\'{e}tique}, Le Math\'{e}maticien, No. 2,
  Presses Universitaires de France, Paris, 1977, Deuxi\`eme \'{e}dition revue
  et corrig\'{e}e. \MR{0498338}

\bibitem[Tot16]{To16}
B.~Totaro, \emph{Hypersurfaces that are not stably rational}, J. Amer. Math.
  Soc. \textbf{29} (2016), no.~3, 883--891. \MR{3486175}

\bibitem[Tse33]{Ts33}
C.~C. Tsen, \emph{Divisionsalgebren \"uber funktionenk\"orpern}, Nachrichten
  von der Gesellschaft der Wissenschaften zu G\"ottingen,
  Mathematisch-Physikalische Klasse \textbf{1933} (1933), 335--339.

\bibitem[War35]{War35}
E.~Warning, \emph{Bemerkung zur vorstehenden {A}rbeit von {H}errn {C}hevalley},
  Abh. Math. Sem. Univ. Hamburg \textbf{11} (1935), no.~1, 76--83. \MR{3069645}

\end{thebibliography}

\end{document}